\documentclass{amsart}

\usepackage{mathtools}

\usepackage{amsfonts, amsthm, amssymb, amsmath, stmaryrd}
\usepackage{mathrsfs,array}
\usepackage{calligra}
\usepackage{eucal,color,times,enumerate,accents}
\usepackage[all]{xy}
\usepackage{url}
\usepackage{enumitem}
\usepackage{enumerate}
\usepackage[pdftex,bookmarksnumbered,bookmarksopen]
{hyperref}
\usepackage{lipsum}
\usepackage{adjustbox}
\usepackage{cite}
\usepackage{stix}
\usepackage[colorinlistoftodos]{todonotes}
\usepackage{tikz-cd}
\usetikzlibrary{decorations.pathreplacing}
\usetikzlibrary{positioning}
\usepackage{dsfont}
\usepackage{tikz}
\usetikzlibrary{matrix,arrows,decorations.pathmorphing,positioning}
\usepackage{calligra}
\usepackage[colorinlistoftodos]{todonotes}
\usepackage{xy}
\xyoption{all}
\usepackage{tikz}
\usetikzlibrary{matrix,arrows,decorations.pathmorphing}
\usepackage{leftidx}
\input xy
\xyoption{all}

\usepackage{tikz-cd}
\theoremstyle{plain}
\newtheorem{thm}{Theorem}[section]
\newtheorem{conj}{Conjecture}[section]

\newtheorem{lemma}[thm]{Lemma}
\newtheorem*{thm*}{Theorem}
\newtheorem*{prop*}{Proposition}
\newtheorem{prop}[thm]{Proposition}
\newtheorem{cor}[thm]{Corollary}
\theoremstyle{definition}
\newtheorem{defi}[thm]{Definition}

\theoremstyle{remark}
\newtheorem{rmk}{Remark}[section]
\newtheorem{rmks}{Remarks}[section]

\numberwithin{equation}{section}

\newcommand{\divides}{\bigm|}
\newcommand{\overbar}[1]{\mkern 1.5mu\overline{\mkern-1.5mu#1\mkern-1.5mu}\mkern 1.5mu}

\DeclareMathOperator{\coker}{coker}

\DeclareMathOperator{\Br}{Br}

\DeclareMathOperator{\End}{End}

\DeclareMathOperator{\Pic}{Pic}
\DeclareMathOperator{\Hom}{Hom}
\DeclareMathOperator{\Gal}{Gal}

\DeclareMathOperator{\Cl}{Cl}
\DeclareMathOperator{\Spec}{Spec}

\DeclareMathOperator{\Res}{Res}

\DeclareMathOperator{\Id}{Id}

\DeclareMathOperator{\rec}{rec}
\DeclareMathOperator{\art}{art}

\DeclareMathOperator{\tr}{tr}
\DeclareMathOperator{\GL}{GL}

\DeclareMathOperator{\Gl}{GL}

\DeclareMathOperator{\U}{U}

\newcommand{\DT}{\mathbb{S}}

\newcommand{\id}{\operatorname{Id}}

\newcommand{\MT}{\operatorname{MT}}
\newcommand{\Hdg}{\operatorname{Hdg}}
\newcommand{\Aut}{\operatorname{Aut}}
\newcommand{\NS}{\operatorname{NS}}

\newcommand{\Z}{\mathbb{Z}}
\newcommand{\Q}{\mathbb{Q}}

\newcommand{\R}{\mathbb{R}}
\newcommand{\Pp}{\mathbb{P}}
\newcommand{\A}{\mathbb{A}}

\newcommand{\Oo}{\mathcal{O}}

\newcommand{\D}{\mathcal{D}}

\newcommand{\p}{\mathfrak{p}}

\newcommand{\et}{\textrm{\'{e}t}}

\newcommand{\C}{\mathbb{C}}

\newcommand{\ord}{\mathrm{ord}}

\newcommand{\adjunction}[4]{\xymatrix@1{#1{\ } \ar@<-0.3ex>[r]_{ {\scriptstyle #2}} & {\ } #3 \ar@<-0.3ex>[l]_{ {\scriptstyle #4}}}}

 \usepackage[foot]{amsaddr}

\begin{document}

\title{Complex multiplication and Brauer groups of K3 surfaces}

\author{Domenico Valloni}
\address{Imperial College London, South Kensington, London SW7 2BU}
\email{d.valloni16@imperial.ac.uk}
\maketitle
\begin{abstract}
We study K3 surfaces with complex multiplication following the classical work of Shimura on CM abelian varieties. After we translate the problem in terms of the arithmetic of the CM field and its id\`{e}les, we proceed to study some abelian extensions that arise naturally in this context. We then make use of our computations to determine the fields of moduli of K3 surfaces with CM and to classify their Brauer groups. More specifically, we provide an algorithm that given a number field $K$ and a CM number field $E$, returns a finite lists of groups which contains $\text{Br}(\overline{X})^{G_K}$ for any K3 surface $X/K$ that has CM by the ring of integers of $E$. We run our algorithm when $E$ is a quadratic imaginary field (a condition that translates into $X$ having maximal Picard rank) generalizing similar computations already appearing in the literature. 
\end{abstract}

\tableofcontents

\section{Introduction}
As it was shown by Shimura in his seminal work \cite{MR1492449} one can study abelian varieties with CM, their torsion points, and their polarizations, only in terms of arithmetic data on their field of complex multiplication. As a matter of fact this idea can be applied to every Hodge structure with abelian Mumford-Tate group, and our aim is to study K3 surfaces with complex multiplication from this point of view. 
\begin{defi}  \label{def of cm}
A K3 surface $X / \C$ has CM if the Mumford-Tate group of $\mathrm{H}_B^2(X, \Q)$ is abelian. 
\end{defi}
Let $\NS(X) \subset \mathrm{H}^2_B(X, \Z)(1)$ be the N\'{e}ron-Severi group of $X$ and let $T(X) = \NS(X)^{\perp}$ be the lattice of transcendental cycles. The latter is an integral Hodge structure of type $ \{ (1,-1), (0,0),(-1,1) \}$ that does not contain any non-trivial sub-Hodge structure of smaller rank, at least when $X$ is projective. Zarhin in \cite{MR697317} showed that Definition \ref{def of cm} is equivalent to the following two properties of $T(X)$:
\begin{enumerate}
\item $\End_{\Hdg}(T(X)_\Q) \cong E,$ a CM field, and
\item $\dim_E T(X)_\Q = 1$, i.e. $[E \colon \Q] = \dim_\Q T(X)_\Q.$
\end{enumerate}
Therefore, complex multiplication can be read from the transcendental lattice of $X$ and, since $\dim_\Q (T(X)_\Q) \leq 20$, we always have $[E \colon \Q] \leq 20$ by point (2) in the above definition. Following the results of Rizov \cite{2005math......8018R} (see also Corollary 4.4 of Madapusi Pera's paper \cite{MR3370622}) we know, loosely speaking, that the Galois action on K3 surfaces with CM and their \et ale cohomology groups is the one predicted by Deligne in the definition of the canonical models of Shimura varieties. This is the analogue of the main theorem of complex multiplication for K3 surfaces (Section \ref{section main thm}) and it constitutes the fundamental building block of our work. Note that it implies that every complex K3 surface with CM is defined over $\overline{\Q}$, a fact already known to Shafarevich. We define now the class of surfaces we work with.
\begin{defi}
Let $X/\C$ be a K3 surface with CM. Then $X$ is \textit{principal} if $\End_{\Hdg}(T(X))$ is the maximal order of the CM field $E   \coloneqq \End_{\Hdg}(T(X)_\Q)$. We also say that $X$ has CM by $\Oo_E$, meaning the same thing. 
\end{defi}
Note that this definition is borrowed from the case of elliptic curves, and it is a classical fact that isomorphism classes of principal elliptic curves with CM by $E$ are parametrized by the class group of $E$; in particular, they are only finitely many. It is interesting to notice that for K3 surfaces things are different: in Proposition \ref{key proposition} we show how a result of Nikulin together with the surjectivity of the period map imply that for any given CM number field $E$ with $2 \leq [E \colon \Q] \leq 10$ there are \textit{infinitely} many (non-isomorphic) principal K3 surfaces with CM by $\Oo_E$. Also, when $10 < [E \colon \Q] \leq 20$, the existence or the infinitude of K3 surfaces with CM by $\Oo_E$ can be stated in purely lattice-theoretical terms, as shown in the proof of Proposition \ref{key proposition}. 

\subsection{Examples of CM K3 surfaces}
The first examples of K3 surfaces satisfying Definition \ref{def of cm} occur when $X / \C$ has maximal Picard rank $\rho(X) = 20$. These are named singular (or exceptional) K3 surfaces and always have CM by an imaginary quadratic field, generated by the square root of the discriminant of $T(X)$ (in Section \ref{sectionsingular} one can find conditions on $T(X)$ to ensure that $X$ is principal). Their geometry was studied by Shioda and Inose in \cite{MR0441982}, who related them to CM elliptic curves with a construction now known as Shioda-Inose structure, whereas their arithmetic properties (fields of definition, classification over $\Q$, relations to binary quadratic forms and relations to modular forms) were investigated by Elkies and Sch\"{u}tt ( \cite{MR3046305}, \cite{MR2346573}, and \cite{MR2602669}). Other examples occur when $X$ is the Kummer surface associated to an abelian surface with CM. Aside from these two classes, one can try to find K3 surfaces $X$ for which the action of $\Aut(X)$ on $T(X)_\Q$ generates a field $E$ for which $\dim_E T(X)_\Q =1$. Then, $X$ has CM and $E$ will always be a cyclotomic field, because $\Aut(X)$ acts on $\mathrm{H}^{2,0}(X, \C)$ via roots of unity (see \cite{Livne2009TheMO} for explicit examples). In all these constructions we note that
\begin{enumerate}
\item $[E \colon \Q] = \dim_\Q(T(X)_{\Q}) \in \{2,4\}$ or $E$ is cyclotomic;
\item The CM action can be constructed geometrically, either from abelian varieties with CM or from automorphisms. 
\end{enumerate} 
On the other hand, Taelmann \cite{taelman2016} showed that for any CM field $E$ with $2 \leq [E \colon \Q] \leq 20$ there are infinitely many $\C$-isomorphism classes of K3 surfaces with CM by $E$. Note however that his proof is purely transcendental, meaning that it constructs K3 surfaces from Hodge theory using the surjectivity of the period map. Therefore, we do not know much about their geometry. In fact, the problem of understanding whether a given K3 surface $X \subset \Pp_\C^n$ has CM feels as difficult as understanding whether a given curve $C \subset \Pp^n_\C$ has CM Jacobian. As far as we know, the only results toward \textit{real} multiplication that are not explained via (2) are the ones of Elsenhans and Jahnel \cite{Elsenhans2014}, who were able to prove the non-triviality of $\End_{\Hdg}(T(X)_\Q)$ for an explicit one-dimensional family of K3 surfaces, of generic Picard rank $16$, by counting points of their reductions $\text{mod} \,p$. 
\subsection{Brauer groups and our results} Since studying rational points of a K3 surface $X$ over a number field $K$ is often difficult, one can always attempt to compute the Brauer-Manin obstruction first: if $\A_K$ denotes the ad\`{e}les of $K$, Manin showed that there exists a natural pairing (described in Section \ref{SectionBrauer}) $$X(\A_K) \times \Br(X) \rightarrow \Q/ \Z,$$
such that $$X(K) \subset X(\A_K)^{\Br(X)} \coloneqq \{ x \in X(\A_K) \colon (x, \alpha) = 0 \,\, \forall \,\, \alpha \in \Br(X) \}. $$ In \cite{MR2664991} Skorobogatov made the following conjecture. 
\begin{conj}[Skorobogatov]
Let $X$ be as before and let $\overline{X(K)}$ be the closure of $X(K)$ in $X(\A_K)$ with respect to the ad\`{e}lic topology. Then $\overline{X(K)} = X(\A_K)^{\Br(X)}.$
\end{conj}
Let $\Br_0(X) \subset \Br(X)$ denote the constant classes of $\Br(X)$, the ones that come from the pullback of the structural morphism $X \rightarrow \Spec(K)$. Then the group $\Br(X) / \Br_0(X)$ is finite for K3 surfaces by the results of Zarhin and Skorobogatov \cite{MR2395136} and, as showed by Kresch and Tschinkel, in order to compute $X(\A_K)^{\Br(X)}$ efficiently one only needs to bound the order of $\Br(X) / \Br_0(X).$
\begin{thm*}[Kresch and Tschinkel, \cite{KRESCH20114131}] \label{Tschinkel}
Let $X/K$ be as above, assume that $X$ is given as a system of homogeneous equations in some projective space, that explicit generators for $\NS(\overline{X})$ are known and that $|\Br(X) / \Br_0(X)|$ can be effectively bounded. Then $X(\A_K)^{\Br(X)}$ is effectively computable (meaning that there exists an algorithm that returns $X(\A_K)^{\Br(X)}$, with an explicit bound on the running time).   
\end{thm*}
As we shall explain in Section \ref{SectionBrauer}, there is another subgroup $\Br_0(X) \subset \Br_1(X) \subset \Br(X)$ such that 
\begin{itemize}
\item $\Br_1(X) / \Br_0(X) \cong \mathrm{H}^1(G_K, \NS(\overline{X})) $;
\item $\Br(X) / \Br_1(X) \subset \Br(\overline{X})^{G_K}$;
\end{itemize}
where $\overline{X} = X \times \Spec(\overline{K})$ and $G_K$ is the absolute Galois group of $K$. In practice, bounding $| \Br(X) / \Br_0(X)|$ reduces to studying $\mathrm{H}^1(G_K, \NS(\overline{X}))$ and $\Br(\overline{X})^{G_K}$, the second group being usually the harder to understand. This kind of problem has been studied for some particular K3 surfaces, most of them having CM or being Kummer of a product of two elliptic curves. Ieronymou, Skorobogatov and Zarhin in \cite{MR3590544, MR2920883, MR2802504} have studied the cases when $X$ is a Kummer surface associated to a product of two elliptic curves, or a diagonal quartic surface defined over $\Q$. Newton's paper \cite{MR3483120} contains general results when $X$ is the Kummer surface associated to the self product of an elliptic curve with CM (in particular, $X$ has maximal Picard rank). Her approach is quite similar to ours, in the sense that we also use class field theory in a crucial way. V\'{a}rilli-Alvarado and Viray \cite{MR3731278} have studied the boundness of $\Br(X) / \Br_1(X) $ for some particular Kummer surfaces, and they proved the existence of a bound (Theorem 1.3 and 1.5 of \textit{loc. cit.}) when restricting to $\ell$-torsion $ \big( \Br(X) / \Br_1(X) \big)[\ell^\infty]$, with $\ell$ an odd prime number. Finally, a similar question was studied by Cadoret and Charles in \cite{cadoret2018remark}. Given a prime number $\ell$, they prove the existence of a universal bound for $\Br(\overline{X})^{G_K}[\ell^\infty]$ when $X$ is allowed to vary in a one-dimensional family (see Theorem 1.2.1 in \textit{loc. cit.} for a precise statement). \\ 
The next theorem is our main result in this direction and, as explained in Section \ref{section applications brauer}, it follows formally from the theory developed in the previous sections. 
\begin{thm} 
 \label{Brauer theorem}
There is an algorithm that, given as an input a CM number field $E$ and a number field $K$, returns a \textit{finite} list of groups $\Br(E,K)$ such that for every K3 surface $X/K$ with CM by $\Oo_E$, $$\Br(\overline{X})^{G_K} \in \Br(E,K).$$
\end{thm}
Consider for example $K=E=\Q(i)$. Running the algorithm above we found that $\Br(\Q(i), \Q(i))$ consists of   $$0,  \, \, \Z /2 , \, \,  (\Z /2)^2 , \,\, \Z / 4  \times \Z /2 , \,\, (\Z /4)^2 , \,\, \Z / 8  \times \Z /4 , $$
 $$(\Z /8)^2 , \,\, (\Z / 3)^2 ,  \, \,  (\Z / 3)^2 \times \Z /2 , \,\, (\Z / 3)^2 \times (\Z /2)^2 ,$$ $$(\Z / 5)^2, \,\, (\Z / 5)^2  \times \Z /2 , \,\, (\Z / 5)^2  \times (\Z /2)^2.$$ If one were interested, for instance, in computing the transcendental Brauer-Manin obstruction for a diagonal quartic surface $X_{a,b,c}/ \Q$ given by the equation $x^4 + ay^4 + bz^4 + cw^4 = 0,$ then one would automatically know that $$\Br( \overline{X_{a,b,c}})^{G_{\Q}} \subset \Br(\overline{X_{a,b,c} })^{G_{\Q(i)}} \in \Br(\Q(i),\Q(i)),$$
making the computations effective for every parameter $a,b,c \in \Q$. Note moreover that diagonal quartics are all isomorphic to the Fermat hypersurface $x^4 + y^4 + z^4 + w^4 = 0$ over an algebraic closure of $\Q$. On the other hand, there are infinitely many $\C$-isomorphism classes of K3 surfaces with CM by $\Z[i]$ (that can be explicitly constructed using Shioda-Inose ideas in \cite{10.1007/978-3-642-20300-8_15}) and we remark that the list above works for each one of them that admits a model over $\Q(i)$. 
\subsection{Strategy of the proof and fields of moduli}
As already mentioned, to build the algorithm we shall make use of mainly two ingredients: the main theorem of complex multiplication for K3 surfaces and the ad\`{e}lic language developed by Shimura. In Sections \ref{Section6}, \ref{Section7} and \ref{Section8} we explain how to adapt Shimura ideas to K3 surfaces. Then in Section \ref{section K3 class groups} we associate to any ideal $I \subset \Oo_E$ an abelian field extension $F_I / E$. By class field theory, the extension $F_I/ E$ is determined by a finite-index subgroup $S_I \subset \A_{E,f}^\times$ inside the finite id\`{e}les of $E$ and we have $$S_I   \coloneqq \{ s \in \A^\times_{E,f} \colon \exists \,\, e \in E^\times \,\, \text{such that} \,\, \frac{se}{\overbar{se}} \in \hat{\Oo}^\times_E \,\, \text{and} \,\, \frac{se}{\overbar{se}} \equiv 1 \mod I \}.$$ 
It follows from the formula above that $S_I = S_{\overline{I}} = S_{I \cap \overline{I}}$, so we can assume without loss of generality that $I \subset \Oo_E$ is such that $I = \overline{I}$. In sections \ref{section K3 class groups} and \ref{section cardinality K3 class groups} one can find a detailed study of these field extensions and a closed formula for the degrees $[F_I \colon E]$. When $E$ is quadratic imaginary we can describe $F_I$ as follows: denote by $K_I$ and $\text{Cl}_I$ respectively the ray class field and the ray class group of $E$ modulo $I$, so that $K_I / E$ is an abelian extension with Galois group isomorphic to $\text{Cl}_I$. Then $$E \subset F_I \subset K_I$$ is the fixed field of $\{ x \in \text{Cl}_I \colon x = \overline{x} \} \subset \text{Cl}_I$ (where complex conjugation acts on $\text{Cl}_I$ due to $I = \bar{I}$). Note that if $X / \C$ has CM by $\Oo_E$ there is a natural action of $\Oo_E$ on $\Br(X)$, and we denote $\Br(X)[I]   \coloneqq \{ \alpha \in \Br(X) \colon i \alpha = 0  \,\, \forall \,\, i \in I\}$. The meaning behind the definitions above lies in the next theorem, which is also the main technical result of the paper. 

\begin{thm} \label{field of moduli IN}
Let $X/ \C$ be a K3 surface with CM by $\Oo_E$. Then 
\begin{enumerate}
\item The field extension $F_I/E$ corresponds to the fixed field of  
$$ \,\, \,\,\,\,\, \,\,\,\,\, \,\,\,\,\, \,\, \{ \sigma \in \Aut(\C / E) \colon \exists \,\, \text{Hodge isometry} \,\, f \colon T(X) \rightarrow T(X^\sigma)\colon f^* \circ \sigma^*|_{\Br(X)[I]} = \id \}, $$
where $X^\sigma = X \times_\sigma \Spec(\C)$, $\sigma^* \colon \Br(X) \rightarrow \Br(X^\sigma)$ is the natural pullback map, and $f^* \colon \Br(X^\sigma) \rightarrow \Br(X)$ the map induced via the identification $\Br(X) \cong \Hom(T(X), \Q/ \Z)$.
Differently said, $F_I$ is the field of moduli of $(T(X), \Br(X)[I])$ over $E$. 
\item If $\rho(X) \geq 12$ (i.e., if $[E \colon \Q] \leq 10$), the field of moduli of $X$ over $E$ corresponds to $F_{\Oo_E}/ E$. In particular, it does not depend on $X$.
\end{enumerate}
\end{thm}
\begin{rmk}  If $[E \colon \Q] \leq 10$ there are infinitely many K3 surfaces with CM by $\Oo_E$ by Proposition \ref{key proposition}, and they all have the same field of moduli. On the other hand, the main result of Skorobogatov and Orr \cite{MR3830546} says that only finitely many of them can be defined over a number field of bounded degree. It follows that the difference between the field of moduli of a K3 surface and a minimal field of definition can be arbitrarily large. 
\end{rmk}
To see how the algorithm in Theorem \ref{Brauer theorem} works, note that if $X$ is defined over a number field $K$ containing $E$ there exists a unique ideal $I \subset \Oo_E$ such that $\Br(\overline{X})^{G_K} = \Br(\overline{X})[I] \cong \Oo_E / I$. Therefore $F_I \subset K$ because of Theorem \ref{field of moduli IN}, and it follows that we can write 
\begin{equation}
\Br(E,K) = \{ \Oo_E / I \colon F_I \subset K \}   \tag{*}
\end{equation}
or less precisely 
\begin{equation} 
    \Br(E,K)' = \{ \Oo_E / I \colon [F_I \colon E] \,\, \text{divides} \,\, [K \colon E]  \}. \tag{**}
\end{equation}
One then uses the explicit formula for $[F_I \colon E]$ (Theorem \ref{invariantformula}) to find all the possible $I \subset \Oo_E$ such that $ [F_I \colon E]$ divides $[K \colon E]$. Note that this strategy is analogous to the one employed by Silverberg to study torsion points on CM abelian varieties in \cite{MR3778184}. We give examples of both approaches: we use (**) to give explicit lists covering the cases when $E= \Q(i), \Q(\sqrt{-3})$ and $K=E$, whereas we use (*) to give a simple criterion for $\Oo_E/ I$ to be a possible Brauer group when $E$ is quadratic imaginary and $K$ is the Hilbert class field of $E$ (Theorem \ref{rachelk}).

\subsection*{Acknowledgments}

The idea of studying Brauer groups of CM K3 surfaces the same way one studies torsion points of CM abelian varieties was suggested to me by my supervisor Alexei Skorobogatov. I am hence most grateful to him, for his infinite patience and many insights and discussions. Without him, this could have not been possible. A special thanks also goes to Martin Orr, who spotted some mistakes in the early drafts and helped me to fix some of them. Finally, I would like to thank Gregorio Baldi, Salvatore Floccari and Matteo Tamiozzo for many stimulating discussions, ideas, and for carefully reading the first drafts of this document.  

\subsection*{Notation}
\begin{itemize}
\item If $K$ is a field, we denote by $\overbar{K}$ a fixed algebraic closure and by $G_K$ its absolute Galois group. For every scheme $X/K$ we write $\overline{X}$ for the base change $X \times_K \overbar{K}$. 
\item We denote by $\A$ the ring of ad\`{e}les over $\Q$ and by $\A_f \subset \A$ the subring of finite ad\`{e}les. Moreover, we denote by $\widehat{\Z} \subset \A_f$ the pro-finite completion of $\Z$, so that $\widehat{\Z} \otimes \Q = \A_f$. \item For any number field $K$, we denote by $\Oo_K$ its ring of integers, by $\A_K   \coloneqq \A \otimes_\Q K$ the ring of ad\`{e}les over $K$ and by $\A_{K,f}   \coloneqq \A_f \otimes_\Q K \subset \A_K$ the subring of finite ad\`{e}les. We also adopt the notation $\widehat{\Oo}_E   \coloneqq \Oo_E \otimes \widehat{\Z}$.
\item In general for any finitely generated abelian group $A$, we denote by $\widehat{A} := A \otimes_\Z \widehat{\Z}$ its profinite completion. We extend this notation also to the cohomology of K3 surfaces, for example, $\widehat{T}(X)$ will denote the profinite completion of the transcendental lattice $T(X)$ of a K3 surface $X/ \C$.
\item If $A$ is a $\Z-$module, we write $A_\Q$ for $A \otimes_\Z \Q$. 
\item For any set $S$, $|S|$ will denote its cardinality, and for any two integers $a,b \in \Z$ we write $a | b$ for `$a$ divides $b$'. 
\item If $A$ is an abelian group and $n$ is an integer, we write $A[n]$ for the $n-$torsion of $A$. 
\item By a lattice we mean a free, finitely generated $\Z$-module $N$ endowed with a symmetric, bilinear, non-degenerate pairing $N \times N \rightarrow \Z$. Its signature is the signature of $N_\R$. 
\end{itemize}

\section{K3 surfaces with CM and their Hodge structures} 
\subsection{Some preliminaries on Hodge theory} 
We begin by reviewing the notion of integral and rational Hodge structures. We mainly follow Moonen's survey \cite{MoonenMT} and Chapter 2 in Milne's notes on Shimura varieties appearing in \cite{Arthur2005HarmonicAT}. The acquainted reader can skip directly to the next subsection. 
\begin{defi} \label{Definition of Hodge structure}
Let $V$ be a finitely generated, free $\Z$-module. An integral Hodge structure of weight $m \in \Z$ on $V$ is a decomposition \begin{equation}
V \otimes_\Z \C = \bigoplus_{p+q =m}V^{p,q}
\end{equation}
such that $\overline{V^{p,q}} = V^{q,p}$. Here, $p$ and $q$ are allowed to vary in $\Z$, and the bar denotes the complex conjugation. One says that the Hodge structure $V$ is of type $T$, where $T \subset \Z^2$, if $V^{p,q} \neq 0$ precisely when $(p,q) \in T$. 
\end{defi}
Similarly, one defines the concept of \textit{rational} Hodge structure. An equivalent definition of Hodge structure is due to Deligne and it is phrased in the language of algebraic groups. The \textit{Deligne torus} is the real algebraic group $\mathbb{S}  \coloneqq \text{Res}_{\C / \R} \mathbb{G}_m$, where `$\text{Res}$' denotes the Weil restriction of scalars, so that $\mathbb{S}(\R) = \C^\times$. The character group $X^*(\mathbb{S})$ is generated by the two characters $z$ and $\overbar{z}$, that act on the $\R$-points of $\mathbb{S}$, respectively, as the identity and the complex conjugation. One also has the following important characters and cocharacters:
\begin{itemize}
\item The \textit{weight cocharacter} $w \colon \mathbb{G}_{m, \R} \rightarrow \mathbb{S}$ given, on $\R-$points, by the natural inclusion $\R^\times \rightarrow \C^\times$;
\item The \textit{Norm character} $\text{Nm} \colon \mathbb{S} \rightarrow \mathbb{G}_{m,\R}$ given by $z \overline{z}$;
\item The cocharacter $\mu \colon \mathbb{G}_{m,\C} \rightarrow \mathbb{S}_\C$ defined to be the only cocharacter such that $\overline{z} \circ \mu = 1$ and $z \circ \mu = \Id$. 
\end{itemize}
It follows that one can define a Hodge structure on $V$ of weight $m \in \Z$ as a morphism of algebraic groups $$h \colon \mathbb{S} \rightarrow \Gl(V)_\R$$ such that $h \circ w \colon \mathbb{G}_{m, \R} \rightarrow \Gl(V)_\R$ is given by $z \mapsto z^{-m} \Id$. In this case, $V^{p,q}$ corresponds to $$\{ v \in V_\C \colon \text{for every} \,\, (z_1,z_2) \in \mathbb{S}(\C) = \C^\times \times \C^\times \,\, \text{one has}\,\, h_\C(z_1,z_2) \cdot v = z_1^{-p} z_2^{-q} v \}.$$ \\
The cohomology groups of smooth projective varieties are always endowed with a Hodge structure thanks to Hodge theory, and in some cases of interest, like abelian varieties or K3 surfaces, the Hodge structure determines the variety itself. If $X$ is a smooth projective variety, there is a natural splitting 
\begin{equation} \label{Hodge decomposition}
 \mathrm{H}^{n}_B(X, \Z) \otimes_\Z \C \cong \bigoplus_{p+q=n}H^{p,q}(X),
\end{equation}
with $$\mathrm{H}^{p,q}(X)   \coloneqq \mathrm{H}^q(X, \Omega^p_X),$$
where $ \mathrm{H}^{n}_B(X, \Z) $ denotes the $n$-th Betti (or singular) cohomology group of $X$. 

A morphism between two Hodge structures $V$ and $W$ is a $\Z$-linear map $f \colon V \rightarrow W$ such that $f_\C \colon V_\C \rightarrow W_\C$ maps $V^{p,q}$ to $W^{p,q}$. The definition implies that in order for a morphism to exist  $V$ and $W$ must have the same weight (one can also define weighted morphisms to obviate this problem). A \textit{sub-Hodge structure} $W \subset V$ is an inclusion of $\Z$-modules $W \hookrightarrow V$ that is also a morphism of Hodge structures. Usually, the map  $W \hookrightarrow V$ is primitive, i.e., the quotient $V/W$ is torsion-free. If $V$ is a Hodge structure of weight $n$, then the dual $V^\vee = \Hom(V,\Z)$ has a natural Hodge structure of weight $-n$. Similarly, if $V$ and $W$ are two Hodge structures of weight $n$ and $m$ respectively, then also $V \otimes_\Z W$ admits a natural Hodge structure of weight $n+m$. In particular, $\Hom(V,W) = V^\vee \otimes_\Z W$ is a Hodge structure of weight $m-n$. Some trivial but extremely important Hodge structures are given by the \textit{Tate-twists}. These are denoted by $\Z(n)$, with $n \in \Z$, and consists of the $\Z$-module $(2 \pi i)^n \Z \subset \C$ endowed with the only Hodge-structure of type $(-n,-n)$. Tate-twists allow one to shift the weight of Hodge structures, in the sense that if $V$ is a integral Hodge structure of weight $m$, then $V(n)   \coloneqq V \otimes_\Z \Z(n)$ is an integral Hodge structure of weight $m-2n$. Similarly, one can define $\Q(n)   \coloneqq \Z(n) \otimes \Q$. The "$(2 \pi i)$" in the definition comes from the exponential sequence \begin{equation}
0 \rightarrow (2 \pi i) \Z \rightarrow \C \xrightarrow{exp} \C^\times \rightarrow 0,
\end{equation} 
and plays a role mostly when computing periods.
\begin{defi}(Hodge classes)
Let $V$ be a Hodge structure of weight $0$. The space of Hodge classes of $V$ is $$\text{Hdg}(V)   \coloneqq V \cap V^{0,0}.$$
\end{defi}
If $X/ \C$ is a smooth projective variety and $\text{CH}^n(X)$ is its Chow group of codimension-$n$ cycles then the cycle class map $$\mathrm{ch}_n \colon \text{CH}^n(X) \rightarrow \mathrm{H}^{2n}(X,\Z)(n),$$
naturally lands in the space of Hodge classes of $ \mathrm{H}^{2n}(X,\Z)(n)$. When $n=1$, we have that $\text{CH}^1(X) = \Pic(X)$, and Lefschetz proved that $c_1(\text{CH}^1(X)) = \text{Hdg}(\mathrm{H}^2(X,\Z)(1))$. The image $c_1(\text{CH}^1(X))$ is the N\'{e}ron-Severi group of $X$ and it is denoted by $\NS(X)$. As firstly showed by Atiyah and Hirzebruch, the equality $\mathrm{ch}_n(\text{CH}^n(X)) = \text{Hdg}(\mathrm{H}^{2n}(X,\Z)(n))$ does not need to hold when $n>1$, but this is mostly due to primitivity issues, and in fact one has the following conjecture.
\begin{conj}[Hodge conjecture]
 For any $X$ and any $n$ as above one has $$\mathrm{ch}_n(\text{CH}^n(X)) \otimes_\Z \Q= \mathrm{Hdg}(\mathrm{H}^{2n}(X,\Q)(n)).$$
\end{conj} 
The last two notions in Hodge theory that we introduce are polarizations and Mumford-Tate groups. One defines first the Weil operator.
\begin{defi}(Weil operator)
Let $V$ be a Hodge structure, the Weil operator is the morphism $C \colon V_\C \rightarrow V_\C$ given by multiplication by $i^{p-q}$ on $V^{p,q}.$ Since $\overline{V^{p,q}} = V^{q,p}$, one can check that $C$ respects $V_\R$, i.e., it is defined over $\R$. Moreover, if the Hodge structure is given by $h \colon \mathbb{S} \rightarrow \Gl(V)_\R$, then $C = h(i)$. 
\end{defi}
Note that $C^2 = (-1)^m$, where $m$ is the weight of $V$. \begin{rmk} \label{remark Weil operator}
The Weil operator commutes with morphisms of Hodge structures, in the sense that if $f \colon V \rightarrow W$ is a morphism of Hodge structures, then $f \circ C_V = C_W \circ f$, where $C_V$ and $C_W$ denote, respectively, the Weil operator on $V$ and $W$.
\end{rmk} 
\begin{defi}
Let $V$ be an integral Hodge structure of weight $m$. A polarization on $V$ is a morphism of Hodge structures $$\phi \colon V \otimes V \rightarrow \Z(-m)$$
such that the bilinear form on $V_\R$ given by $(x,y) \mapsto (2 \pi i)^m \phi(Cx \otimes y)$ is symmetric and positive-definite.  
\end{defi}
The Hodge structures coming from smooth, projective varieties always admit, usually many, polarizations. Finally, the Mumford-Tate group attached to a Hodge structure can be defined in two different ways, either via the formalism of Tannakian categories, or in more down-to-earth terms. We prefer this latter option, and refer the reader to the relevant article by Deligne in \cite{MR654325} for an introduction to Tannakian categories and related concepts.
\begin{defi}
Let $V$ be a rational Hodge structure given by the morphism $h \colon \mathbb{S} \rightarrow \Gl(V)_\R$. The Mumford-Tate group of $V$, denoted by $\text{MT}(V)$, is defined to be the smallest algebraic subgroup of $\Gl(V)$ such that $h$ factorizes as $h \colon \mathbb{S} \rightarrow \MT(V)_\R \hookrightarrow \Gl(V)_\R.$ 
\end{defi}
Note that $\MT(V)$ is connected since $\mathbb{S}$ is connected and, moreover, if $V$ is a polarization, then $\MT(V)$ is reductive (see Proposition 4.9. in Moonen's notes). Mumford-Tate groups allow us to detect sub-Hodge structures in tensor constructions: let $\lambda \subset \Z^2$ be a finite subset, $\lambda = \{ (a_i,b_i) \}_{i=1, \cdots, n}$, and define $$V^\lambda   \coloneqq \bigoplus_{i=1}^n V^{\otimes a_i} \otimes (V^\vee)^{b_i}.$$ We have a natural action of $\MT(V)$ on $V^\lambda$. 
\begin{prop}
A rational subspace $W \subset V^\lambda$ is a sub-Hodge structure if and only if it is invariant under the action of $\MT(V)$. Moreover, an element $t \in V^\lambda$ is a Hodge class if and only if it is fixed by $\MT(V)$. 
\end{prop}
\begin{defi} \label{Special HS}
Let $V$ be a rational Hodge structure. Following Milne, we say that $V$ is \textit{special} if its Mumford-Tate group is a torus. 
\end{defi}
This definition is very similar to Definition \eqref{def of cm}. The only differences are some technical conditions that are automatically satisfied for K3 surfaces, but need to be imposed for general Hodge structures (see Definition 12.5 in Milne's notes). Let $V$ be a special Hodge structure and let $T$ be its Mumford-Tate group, that by definition is an algebraic torus defined over $\Q$. The cocharacter $\mu$ introduced before gives us a morphism of algebraic tori $\mu' \colon \mathbb{G}_{m,\C} \rightarrow T_\C$. 
\begin{defi} \label{Definition reflex field}
Let $h \colon \mathbb{S} \rightarrow \Gl(V)_\R$ be a special Hodge structure, and let $T$ be its Mumford-Tate group. The reflex field of $V$, denoted by $E(h)$, is the field of definition of the cocharacter $\mu' \colon \mathbb{G}_{m,\C} \rightarrow T_\C$.
\end{defi}
\subsection{K3 surfaces with complex multiplication}
 Let $X/\mathbb{C}$ be a projective $K3$ surface and let $\mathrm{H}^{2}_{B}(X, \mathbb{Z})$ be its second Betti cohomology group. The topological intersection form $$\mathrm{H}^2_B(X,\Z) \times \mathrm{H}^2_B(X,\Z) \rightarrow \Z$$ turns $\mathrm{H}^2_B(X,\Z)$ into a lattice, that is unimodular by Poincar\'{e} duality. Moreover, it follows from the Hodge index theorem that its signature is $(3_+,19_-)$. The isomorphism class of this lattice does not depend on the chosen $X$, since every two K3 surfaces are deformation equivalent (Chapter 7, Theorem 1.1. of \cite{MR3586372}); it is usually denoted by $\Lambda$ and named the \textit{K3 lattice}. This cohomology group naturally carries a Hodge structure of weight $2$, but for our purposes it is more natural to work with the twist $\mathrm{H}^2_B(X,\Z)(1)$, of weight zero. The transcendental lattice of $X$, denoted by $T(X)$, is defined as the orthogonal complement of $\NS(X)$ with respect to the intersection form on $\mathrm{H}^{2}(X, \mathbb{Z})(1)$. Therefore, $T(X)$ is a sub Hodge structure of weight zero, and the embedding $T(X) \hookrightarrow \mathrm{H}^2(X,\Z)(1)$ is primitive. Moreover, one can show that $T(X)_\Q$ is an irreducible rational Hodge structure, at least when $X$ is projective.
\begin{defi} 
We say that $X$ has complex multiplication (CM) if the Mumford-Tate group $\MT(X)$ of $T(X)_{\Q}$ is abelian .
\end{defi}
\begin{rmk}
It is easy to show that the inclusion $T(X)_\Q \subset \mathrm{H}^{2}(X, \Q)(1)$ induces an identification between the Mumford-Tate group of $T(X)_\Q$ and the one of $\mathrm{H}^{2}(X, \Q)(1)$.
\end{rmk}
In this case (see Zarhin \cite{MR697317}) one can prove that $E(X)   \coloneqq \End_{\Hdg}(T(X)_{\Q})$ is a CM field (where complex conjugation acts like the adjunction with respect to the intersection form) and that $\dim_{E(X)}T(X)_{\Q} = 1$. 
Since the elements of $E(X)$ are endomorphisms of Hodge structures, we obtain a natural map $\sigma_X \colon E(X) \rightarrow \End(\mathrm{H}^{1,-1}(X)) = \C$. Since $T(X)_\Q$ is irreducible, Schur's lemma shows that $\sigma_X $ is actually an embedding. Therefore, $E(X)$ is always naturally a subfield of $\C$, and in Proposition  \ref{reflex} we show that it corresponds to the reflex field of the Hodge structure $T(X)_\Q$. 
The Hodge structure $T(X)_\Q$ can be described using the torus $\Res_{E(X) / \Q} \mathbb{G}_{m}$, whose $\Q$-points are naturally identified with $E(X)^\times$. If we decompose $$(\Res_{E(X) / \Q} \mathbb{G}_{m}) (\C ) = \bigoplus_{\sigma : E(X) \hookrightarrow \C}  \C^{\times}_{\sigma}$$
where 
$$ \C^{\times}_{\sigma}   \coloneqq \{ z \in (\Res_{E(X) / \Q} \mathbb{G}_{m}) (\C ) \, \colon \, \forall e \in E(X), \, e \cdot z = \sigma(e)z  \} $$
we have that the Hodge structure on $T(X)_{\Q}$ is given by the morphism of algebraic groups (defined over $\R$) whose action on $\C$-points is 
\begin{align*}
h \colon \mathbb{S}(\C) \cong \C^{\times} \times \C^{\times} &\to \C_{\sigma_X}^{\times}  \oplus  \cdots  \oplus  \C^{\times}_{\overline{\sigma_X}} = \Res_{E(X) / \Q} \mathbb{G}_{m}(\C) \subset \GL(T(X))(\C) \\
(z,w) &\mapsto (  zw^{-1},1, \cdots,1, wz^{-1} ),
\end{align*} 
where $\mathbb{S}   \coloneqq  \Res_{\C / \R} \mathbb{G}_{m}$ is the Deligne torus and $\sigma_X$ is the distinguished embedding $E(X) \hookrightarrow \C$. Denote by $\U_{E(X)}$ the $E(X)$-linear unitary subgroup of $\Res_{E(X) / \Q} \mathbb{G}_{m}$, i.e. the one cut out by the equation $e \bar{e} = 1$. Zarhin in his paper \cite{MR697317} proved that inside $\GL(T(X))_{\Q}$ we have an identification $$\MT(T(X)) = \U_{E(X)}.$$
When taking $\C$-points, the natural inclusion $\U_{E(X)} \subset \Res_{E(X) / \Q} \mathbb{G}_{m}$ becomes $$\U_{E(X)} (\C) = \bigg\{ (z)_{\sigma} \in \bigoplus_{\sigma : E(X) \hookrightarrow \C}  \C^{\times}_{\sigma} \colon z_{\sigma} z_{\overline{\sigma}} =1 \bigg\}.$$
Therefore, the cocharacter $\mu$ associated to $h$ is the map

\begin{align}\label{eqn cocharacter}
\mu \colon \mathbb{G}_{m} (\C) &\to \C_{\sigma_X}^{\times}  \oplus  \cdots  \oplus  \C_{\overline{\sigma_X}}^{\times} \\ \nonumber
z &\mapsto ( z,1, \cdots, 1, z^{-1} )
\end{align}
with image inside $U_{E(X)}(\C)$.
\begin{prop} \label{reflex}
The reflex field of the Hodge structure $T(X)_{\Q}$ is $\sigma_X(E(X)) \subset \C$. 
\end{prop}

\begin{proof}
By definition, the reflex field of $T(X)_{\Q}$ is the field of definition of the cocharacter $\mu$. By the discussion above, we see that $\tau \in \Aut(\C)$ fixes $\mu$ if and only if $\tau \sigma_X = \sigma_X$, i.e. if and only if $\tau \in \Aut(\C / \sigma_X(E(X)))$.  
\end{proof}
\begin{rmks}
\begin{itemize}
\item Note that the reflex field of a CM abelian variety is usually not isomorphic to the CM field $E$ (but it is as soon as $E/ \Q$ is Galois). For example, if the dimension of the variety is high and $E/ \Q$ has no automorphisms, the degree of the reflex field can be of the order of magnitude of $[E \colon \Q]!$.
\item The embedding $\sigma_X$ normalizes the action of $E(X)$ in the sense that if $\alpha \in \sigma_X(E(X))$, then the Hodge endomorphism $\sigma_{X}^{-1}(\alpha)$ acts as multiplication by $\alpha$ on the $(1,-1)$ part of cohomology. 
\end{itemize}
\end{rmks}

One can show that a CM field $E$ can be spanned, as $\Q$-vector spaces, by elements $\alpha \in E$ such that $\alpha \overline{\alpha} =1$ (for a proof, see Proposition 4.4. in  \cite{HuyCM}). In $E(X)$, these correspond to rational Hodge isometries, since for every $v,w \in T(X)_{\Q}$ we have $$(\alpha v,\alpha w)_{X} = (\alpha \overline{\alpha} v, w)_{X} = (v,w)_X.$$

As proved by Buskin in \cite{Buskin}, for any $\alpha \in E$ such that $\alpha \overline{\alpha} = 1$ there exist integral algebraic cycles $C_{i} \subseteq X \times X$ and rational numbers $ q_{i} \in \Q $ for $i=1, \cdots, n$ such that the cohomology class of $\alpha$ in $\mathrm{H}^{4}(X \times X, \mathbb{Q})(2)$ can be expressed as $$\alpha = \sum_{i} q_{i} [C_{i}].$$
(Here, we denote by $[C_i]$ the image of $C_i$ under the cycle class map $\text{CH}^2(X \times X) \rightarrow \mathrm{H}^4(X \times X, \Q)(2)$). Buskin result builds on the previous work of Mukai \cite{mukai} (who proves the same statement but only for K3 surfaces with $\rho(X) \geq 11$) and of Nikulin \cite{NikulinHodgeK3}, who improved Mukai results by comprehending all K3 surfaces with $\rho(X) \geq 5$. Together with the fact that $E$ is spanned by isometries, this implies that the Hodge conjecture is true for $X \times X$, where $X/\C$ is a K3 surface with complex multiplication. In particular, if $X$ is defined over $K\subseteq \C$, one can ask over which extension of $K$ a class $\alpha \in E(X_{\C})$ is defined as well. 
\begin{defi} \label{defined/ K}
Let $X/K \subseteq \C$ (this notation means that $K$ is considered as a subfield of $\C$) with CM over $\C$. 
\begin{enumerate}
\item For every $\tau \in \Aut(\C)$ we define the map $\tau^{ad} \colon E(X_\C) \rightarrow E(X^{\tau}_\C)$ as $$\tau^{ad}(\alpha) : = \sum_{i} q_{i} \tau^*[ C_{i}],$$ where $ E(X_\C) \ni \alpha = \sum_{i} q_{i} [C_{i}] $ and $\tau^{*}$ denotes the pullback of algebraic cycles via the isomorphism of schemes $\tau \colon X_\C^{\tau} \rightarrow X_\C$ (this notation is borrowed from Rizov paper). 
\item We say that $\alpha \in E(X_\C)$ is defined over $K$ if for every $\tau \in \Aut(\C / K)$ $$\tau^{ad} (\alpha ) = \alpha.$$ 
\end{enumerate}
\end{defi}

\begin{defi} \label{CM over K}
Let $X$ be a K3 surface over a field $K$ such that $X_\C$ has CM for an embedding $K \hookrightarrow \C$. We define $E(X)$ to be the subfield of $E(X_\C)$ of endomorphism that are defined over $K$. We say that \textit{$X$ has CM over $K$} if $E(X) = E(X_{\C}).$
\end{defi}
\begin{rmk}
In order to define $E(X)$ one has to choose an embedding $K \hookrightarrow \C$, but one can check that $E(X)$ does not depend on the chosen embedding.
\end{rmk}
We will now give an equivalent condition for $X/K$ to have complex multiplication over $K$, similar to the one for abelian varieties.
\begin{prop} \label{reflex field and CM}
Let $X/K$ be as in Definition \ref{CM over K} such that $X_\C$ has CM, and let $\iota \colon K \hookrightarrow \C$ be an embedding. Then $X$ has CM over $K$ if and only if $$\sigma_{X_\C}(E(X_\C)) \subseteq \iota(K),$$ i.e. if and only if $\iota(K)$ contains the reflex field of $X_\C$. Also, the condition $\sigma_{X_\C}(E(X_\C)) \subseteq~\iota(K)$ does not depend on $\iota$. 
\end{prop}
\begin{proof}
Let $\tau \in \Aut(\C)$ be an automorphism of the complex numbers and consider the base change $X_{\C}^{\tau}   \coloneqq X_{\C} \times_{\tau} \Spec\C$. Again, we have a natural isomorphism $\tau^{\text{ad}} \colon E(X_{\C} ) \xrightarrow{\sim} E(X_{\C}^{\tau})$, given by conjugation of algebraic cycles. If $\omega \in T^{1,-1}(X_{\C})$ is a non-zero $2-$form, we can conjugate it via $\tau$ (since it is an algebraic object) to obtain a non zero $2-$form $\omega^{\tau}$ on $T^{1,-1}(X_{\C}^{\tau})$.
Denote by $\sigma_{X} \colon E(X_{\C}) \hookrightarrow \C$ and by $\sigma_{X^\tau} \colon E(X_{\C}^{\tau}) \hookrightarrow \C$ the two embeddings given by evaluation on a non-zero $2-$form and let $\alpha \in \sigma_{X}(E(X_{\C}))$; we have: $$(\tau^{ad} \sigma_{X}^{-1} \alpha)\omega^{\tau} = ((\sigma_{X }^{-1} \alpha) \omega)^{\tau} = (\alpha \omega)^{\tau} = \tau(\alpha) \omega^{\tau}$$
i.e.
\begin{equation}\label{diagramReflex}
\sigma_{X^{\tau}} \circ \tau^{ad} = \tau \circ \sigma_{X}.
\end{equation}
Meaning that the following diagram commutes
\begin{center}
\begin{tikzcd} 
E(X_{\C} ) \arrow[d, "\sigma_{X}"] \arrow[r, "\tau^{ad}"] & E(X_{\C} ^{\tau}) \arrow[d, "\sigma_{X ^{\tau}}"]  \\
\C \arrow[r, "\tau"] & \C .
\end{tikzcd}
\end{center}
 If $\tau$ fixes $K$, then $X_{\C}^{\tau} = X_{\C}
 $, so that $E(X) = E(X_{\C}) $ if and only if the map $\tau^{ad} \colon E(X_{\C}) \rightarrow E(X_{\C})$ is the identity. But the diagram above tells us that this happens if and only if $\tau$ fixes also $\sigma_{X}(E(X_{\C}))$. Finally, to prove that the condition $\sigma_{X_\C}(E(X_\C)) \subseteq \iota(K)$ does not depend on $\iota$, we need to show that it is true for one embedding if and only if it is true for all. But if $\tau \in \Aut(\C)$ is any element, equation \ref{diagramReflex} implieas that $$ \sigma_{X^{\tau}}(E(X_{\C}^\tau)) = \tau (\sigma_{X}(E(X_\C))),$$ so that we can conclude the proof. 
 \end{proof}

\begin{defi}
Let $X/\C$ be a K3 surfaces with CM. We define the order $\mathcal{O}(X)   \coloneqq \End_{\Hdg}(T(X)) \subset E(X)$, and we say that $X$ is \textit{principal} if $\mathcal{O}(X) $ is the maximal one. 
\end{defi}
\begin{rmk}
From now on, we will only consider $K3$ surfaces with CM that are \textit{principal}. 

\end{rmk}
One has to prove that the ring $\mathcal{O}(X)$ is an algebraic invariant of $X$, i.e. that it depends only on the scheme structure of $X$. What we mean by this is the following: consider $X/k$ any K3 surface, and suppose there exists an embedding $\iota \colon k \hookrightarrow \C$. Base-changing $X$ via $\iota$, we obtain a K3 surface $ X^\iota$ over $\C$, and we can compute the ring $\mathcal{O}(X^\iota) = \End_{\Hdg}(T(X^\iota))$. We need to prove that this ring does not depend on $\iota$. The analogous statement for abelian varieties is trivial, as the analogue of $\End_{\Hdg}(T(X))$ would be the endomorphism ring of the variety, and conjugation of an endomorphism is still an endomorphism. In the $K3$ surface case, though, it is not clear that if $\alpha \in \mathcal{O}(X) \subset E(X)$ then also $\tau^{ad} (\alpha) \in \mathcal{O}(X^{\tau}) \subset E(X^{\tau})$ (we only know, so far, that $\tau^{ad} (\alpha) \in E(X^{\tau})$).

Before stating our next result, note that one can define $\Oo(X)$ and $E(X)$ in the same fashion for any K3 surface $X$, and $E(X)$ naturally lives in the Hodge classes of $\mathrm{H}^4_B(X \times X, \Q)(2).$ As a byproduct of the work of Deligne in (\cite{MR654325}, Theorem 2.11) and \cite{MR0296076} one knows that every Hodge class in a product of K3 surfaces (and abelian varieties) is \textit{absolute Hodge}, a fact that implies that there is always a natural map $\tau^{ad} \colon E(X) \rightarrow E(X^\tau)$ like in \eqref{defined/ K}.
\begin{prop}[Invariance of $\mathcal{O}(X)$] \label{invariance O}
Let $X / \C$ be any K3 surface and let $\tau \in \Aut(\C)$. Then the natural map $\tau^{ad} \colon E(X) \rightarrow E(X^{\tau})$ sends $\mathcal{O}(X)$ isomorphically to $\mathcal{O}(X^{\tau})$.
\end{prop}

\begin{proof}
Consider the two cycle class maps $$\mathrm{ch}_{B} : E(X) \hookrightarrow \Hdg^4(X \times X)(2) \subset \mathrm{H}^{4}_{B}(X \times X, \Q(2)) $$
$$\mathrm{ch}_{\et} : E(X) \hookrightarrow \mathrm{H}^{4}_{\et}(X \times X, \A_{f}(2)).$$
Where for any $k\geq 0$, $\mathrm{H}^{\bullet}_{\et}(- , \A_{f}(k)) = \mathrm{H}^{\bullet}_{\et}(- , \widehat{\Z}(k)) \otimes \Q$ and 
$$ \mathrm{H}^{\bullet}_{\et}(- , \widehat{\Z}(k)) = \varprojlim_n \mathrm{H}^{\bullet}_{\et}(- , \mu_n^{\otimes k})$$
denote the \et ale cohomology groups. 
For every $\tau \in \Aut(\C)$ we have a well-defined map $$\tau_{B} \colon \Hdg^{4}(X \times X) \rightarrow \Hdg^{4}(X^{\tau} \times X^{\tau})$$
due to the fact that every Hodge class is absolute Hodge, and a natural inclusion $$\Hdg^{4}(X \times X) \hookrightarrow \mathrm{H}^{4}_{\et}(X \times X, \A_{f}(2))$$
given by $\Hdg^{4}(X \times X) \hookrightarrow \mathrm{H}^{4}_{B}(X \times X, \Q(2))$ followed by the inclusion $\mathrm{H}^{4}_{B}(X \times X, \Q(2)) \hookrightarrow \mathrm{H}^{4}_{\et}(X \times X, \A_{f}(2)) $ given by the comparison isomorphism  
\begin{equation} \label{Comparison}
    \mathrm{H}^{4}_{B}(X \times X, \Z(2)) \otimes \widehat{\Z} \cong  \mathrm{H}^{4}_{\et}(X \times X, \widehat{\Z}(2)) .
\end{equation}
These maps belong to the following commutative diagram: 
\begin{center}
\begin{tikzcd} 
E(X) \arrow[r, "\tau^{ad}"] \arrow[hookrightarrow, d, "\mathrm{ch}_{B}"]
& E(X^{\tau}) \arrow[hookrightarrow, d, "\mathrm{ch}_{B}"] \\
\Hdg^{4}(X \times X) \arrow[r, "\tau_{B}"] \arrow[hookrightarrow, d,]
& \Hdg^{4}(X^{\tau} \times X^{\tau}) \arrow[hookrightarrow, d,] \\
\mathrm{H}^{4}_{\et}(X \times X, \A_{f}(2)) \arrow[r, "\tau^*"]
& \mathrm{H}^{4}_{\et}(X^{\tau} \times X^{\tau}, \A_{f}(2)).
\end{tikzcd}

\end{center}
where $\tau^{*}$ is the natural pullback in \'{e}tale cohomology via the isomorphism of schemes $\tau \colon X^{\tau} \to X$, and the composition of the vertical arrows is $\mathrm{ch}_{\et}$. Let $\widehat{T}(X) \subset  \mathrm{H}^{2}_{\et}(X, \widehat{\Z}(1))$ be the orthogonal of $\NS(X)$, so that $\widehat{T}(X) \cong T(X) \otimes \widehat{\Z}$ via the comparison isomorphism \ref{Comparison}.
Consider the isomorphism of $\widehat{\Z}$-lattices $$ \tau^* \colon \widehat{T}(X) \rightarrow \widehat{T}(X^\tau) $$ induced by Galois, and let $f \in \mathcal{O}(X)$. The commutativity of the above diagram tells us that $$\tau^{ad}(f) = \tau^* \circ f \circ \tau^{*-1},$$ this equality happening in $\mathrm{H}^{4}_{\et}(X^{\tau} \times X^{\tau}, \A_{f}(2)).$
Now, $\tau^{ad}(f) (T(X^{\tau})) \subset T(X^{\tau})_\Q$ since $\tau^{ad}(f) \in E(X^{\tau})$, and $[\tau^* \circ f \circ \tau^{*-1}] \widehat{T}(X^\tau) \subset \widehat{T}(X^\tau)$ since $ \tau^* \colon \widehat{T}(X) \rightarrow \widehat{T}(X)$ is an isomorphism. Thus, the equality $\tau^{ad}(f) = \tau^* \circ f \circ \tau^{*-1}$ implies that $\tau^{ad}(f) (T(X^{\tau})) \subset  T(X^{\tau})_\Q \cap  \widehat{T}(X^\tau)=  T(X^{\tau})$, i.e. $\tau^{ad}(f) \in \Oo(X^{\tau})$. Hence the map $$\tau^{ad} \colon E(X) \rightarrow E(X^{\tau})$$ restricts to an isomorphism between $\mathcal{O}(X)$ and $\mathcal{O}(X^{\tau})$. 
\end{proof}

\section{Computing the order of singular K3 surfaces} \label{sectionsingular}
In this section we will explicitly compute the order $\mathcal{O}(X)$ for every $X/ \C$ with maximal Picard rank $\rho(X) = 20$, so to have an easy criterion to decide whether $\mathcal{O}(X)$ is principal or not. If $X / \C$ is a singular $K3$ surface, the order $\mathcal{O}(X)  \coloneqq \End_{\Hdg}(T(X))$ can be determined in the following standard way.  Choose a $\Z$-basis $e_{1},e_{2}$ of $T(X)$ and write the intersection matrix as 
\begin{equation} \label{quadratic form}
M=
  \left[ {\begin{array}{cc}
   2a & b \\
   b & 2c \\
  \end{array} } \right]
\end{equation}

with $a,b,c \in \Z$ and $\Delta   \coloneqq b^{2} - 4ac < 0$. Let $2q(x,y)   \coloneqq (x e_{1} + y e_{2}, x e_{1} + y e_{2})_{X}$ be the binary quadratic form associated to $(-,-)_{X}$, i.e. $$q(x,y) = ax^{2} + bxy + cy^{2}.$$
Up to orientation, the only Hodge structure on $T(X)$ of K3 type is given by $$T(X)^{1,-1} = \C 
\begin{bmatrix} s \\ 1 \end{bmatrix} $$
where $s$ is a solution of $q(x,1)=0$, say $s = \frac{-b + \sqrt{\Delta}}{2a}$. This follows by the fact that a non-zero $2$-form $\omega$ must satisfy $q(\omega, \omega) = 0$. Denote by $E$ the field $\Q(\sqrt{\Delta})$ and write $\Delta = f^2 \Delta_E$, with $\Delta_E$ the discriminant of the field $E$. 

\begin{prop}
The ring homomorphism  
\begin{align*}
\Phi \colon E &\rightarrow M_{2 \times 2}(\Q) \\ 
x+y\sqrt{\Delta_E} &\mapsto  x \Id + \frac{y}{f}  \left[ {\begin{array}{cc}
    - b & -2c \\
   2a &  b \\
  \end{array} } \right]
\end{align*}
realizes $E$ as $\End_{\Hdg}(T(X)_{\Q})$
\end{prop} 
 
\begin{proof}
The fact that the above map is a morphism of rings is an easy computation. The only thing left to check is that $\Phi(E) \subset \End_{\Hdg}(T(X)_{\Q})$, and this is equivalent to $\Phi(\sqrt{\Delta_E}) \in \End_{\Hdg}(T(X)_{\Q})$. Now, $$\Phi(\sqrt{\Delta_E}) = \frac{1}{f} \left[ {\begin{array}{cc}
   - b & -2c \\
   2a & b \\
  \end{array} } \right]
$$
and we have 
$$\frac{1}{f}\left[ {\begin{array}{cc}
   - b & -2c \\
   2a & b \\
  \end{array} } \right]  \begin{bmatrix} s \\ 1 \end{bmatrix} =  \frac{1}{f}\left[ {\begin{array}{cc}
   \frac{\Delta - b\sqrt{\Delta}}{2a} \\
   \sqrt{\Delta} \\
  \end{array} } \right] = \sqrt{\Delta_E} \begin{bmatrix} s \\ 1 \end{bmatrix} $$
\end{proof}

\begin{thm} \label{O(X)}
Let $X/ \C$ be a singular K3 surface, let $q(x,y)   \coloneqq ax^2 + bxy + cz^2$ the quadratic form associated to a $\Z$-basis of $T(X)$, of discriminant $\Delta = f^2 \Delta_E$, with $\Delta_E$ the discriminant of the field $E=\Q(\sqrt{\Delta})$. Then 

$$ \mathcal{O}(X) \cong \Z + \frac{f}{(a,b,c)}\mathcal{O}_{E}.$$
In particular, $X$ is principal if and only if $f = (a,b,c)$.
\end{thm}
\begin{proof} 
From the discussion above, we know that the order $\mathcal{O}(X)$ corresponds to $$\mathcal{O}(X) \cong \bigg\{ x,y \in \Q \colon \left[ {\begin{array}{cc}
   x - \frac{b}{f}y & -\frac{2c}{f}y \\
   \frac{2a}{f}y & x + \frac{b}{f}y \\
  \end{array} } \right] \in M_{2 \times 2}(\Z) \bigg\}. $$
  
This is equivalent to $2x \in \Z$, $\frac{2(a,b,c)}{f}y \in  \Z$ and $ x - \frac{b}{f}y \in \Z$, i.e. 

$$\mathcal{O}(X) \cong \bigg\{ \frac{x}{2} + \frac{fy}{2(a,b,c)}\sqrt{\Delta_E} \colon x,y \in \Z \, , \, x + \frac{b}{(a,b,c)}y  \equiv 0 \, \text{mod} \, 2 \bigg\}.$$
We also have 
 $$ \bigg( \frac{b}{(a,b,c)} \bigg)^2 \equiv \bigg( \frac{f}{(a,b,c)} \bigg)^2 \Delta_E \mod 4$$
If $\Delta_E \equiv 0$ mod $4$ then the above equations forces $$\bigg( \frac{f}{(a,b,c)} \bigg)^2 \equiv \bigg( \frac{b}{(a,b,c)} \bigg)^2  \equiv 0 \mod 4$$ and $\mathcal{O}(X)$ corresponds to 
$$\mathcal{O}(X) \cong \bigg\{ \frac{x}{2} + \frac{fy}{2(a,b,c)}\sqrt{\Delta_E} \colon x,y \in \Z \, , \, x \equiv 0 \mod 2 \bigg\} = \Z + \frac{f}{(a,b,c)} \mathcal{O}_{E} $$
If $\Delta_E \equiv 1 \mod 4$ and $\frac{f}{(a,b,c)} $ is odd, the order $\mathcal{O}(X)$ corresponds to  
$$\mathcal{O}(X) \cong \bigg\{ \frac{x}{2} + \frac{fy}{2(a,b,c)}\sqrt{\Delta_E} \colon x,y \in \Z \, , \, x + y \equiv 0 \mod 2 \bigg\} = \Z +  \frac{f}{(a,b,c)} \mathcal{O}_{\Q(\sqrt{\Delta'})}$$
And finally, if $\Delta_E \equiv 1 \mod 4$ and $\frac{f}{(a,b,c)} $ is even, $\mathcal{O}(X)$ corresponds to 
$$\mathcal{O}(X) \cong \bigg\{ x + \frac{fy}{2(a,b,c)}\sqrt{\Delta_E} \colon x,y \in \Z  \bigg\}= \Z + \frac{f}{2(a,b,c)} \bigg( \Z + 2\mathcal{O}_{E} \bigg)= \Z + \frac{f}{(a,b,c)} \mathcal{O}_{E}.$$
\end{proof}

\begin{cor} \label{Infinite pic 20}
Let $E$ be an imaginary quadratic extension of $\Q$. Then there are infinitely many $\C$-isomorphism classes of K3 surfaces with CM by $\Oo_E$. 
\end{cor}
\begin{proof}
As proved in \cite{MR0284440}, K3 surfaces with maximal Picard rank correspond bijectively to isomorphism classes of positive-definite oriented even lattices of rank two, via $X \mapsto T(X)$ (this is indeed very similar to what happens in Proposition \ref{key proposition}). Let $E$ be any imaginary quadratic field and choose a lattice $M$ like $\eqref{quadratic form}$, with $E = \Q(\sqrt{b^2-4ac})$ and $f = (a,b,c)$. Write $X_M$ for the only K3 surface with $T(X_M) \cong M$. By Theorem \ref{O(X)}, $X_M$ has CM by $\Oo_E$. But for every $n \in \Z_{>0}$ also $X_{nM}$ has CM by $\Oo_E$ and $X_{M}$ is not isomorphic to $X_{nM}$ if $n>1$.  
\end{proof}
In Proposition \ref{key proposition} we will extend this result to all $E$ with $[E \colon \Q] \leq 10$. 
\section{Brauer groups}  \label{SectionBrauer}
In this section we recall some theory about the Brauer group and its connection to the transcendental lattice of (K3) surfaces. We refer the interested reader to Section 4.3. of \cite{Alexeibook} for a thorough explanation. Let $K$ be a field of characteristic zero, $\overbar{K}$ a fixed algebraic closure and $G_K$ its absolute Galois group. The Brauer group of a smooth, geometrically integral variety $X / K$ is $\Br(X)   \coloneqq \mathrm{H}^2_{\et}(X, \mathbb{G}_m) $. It is always a torsion abelian group under our assumptions on $X$, and the association $X \mapsto \Br(X)$ is functorial and contravariant in $X$. When $X = \mathrm{Spec}(K)$ then $\Br(K) = \mathrm{H}^2(G_K, \overbar{K}^\times)$ is the classical Brauer group of $K$ and parametrizes central simple algebras (or Severi-Brauer varieties) over $K$, modulo an appropriate equivalence relation. If moreover $K$ is a number field, there is the following explicit description of $\Br(K)$ given by class field theory. One first computes the Brauer groups of the local completions: let $v$ be any place of $K$ and denote by $K_v$ the completion of $K$ at $v$. Then 
\begin{itemize}
\item  If $v$ is finite, there is a canonical isomorphism $\mathrm{Inv}_v \colon \Br(K_v) \cong \Q/ \Z$ given by the \textit{invariant map};
\item $\Br(K_v) = \Br(\R) = \Z/ 2 \Z$ if $v$ is real ;
\item $\Br(K_v) = \Br(\C) =0$ if $v$ is complex.
\end{itemize} 
From $K \subset K_v$ we obtain $\Br(K) \rightarrow \Br(K_v)$, and all these maps fit into the fundamental exact sequence
\begin{equation} \label{ctfBrauer}
 0 \rightarrow \Br(K) \rightarrow \bigoplus_{v} \Br(K_v) \xrightarrow{\sum \textrm{Inv}_v} \Q / \Z \rightarrow 0
\end{equation}
(where we put $\textrm{Inv}_v \colon \Br(K_v) \cong \Z/ 2 \Z \hookrightarrow \Q/ \Z$ if $v$ is real). Returning to positive dimensional subjects, assume that $X/ K$ is also proper. Then its ad\`{e}lic points can be written as $$X(\A_K) = \prod_v X(K_v)$$ because of the valuative criterion for properness, and $X(\A_K)$ is non-empty precisely when $X$ has points locally everywhere. Manin observed in \cite{MR0427322} that \eqref{ctfBrauer} gave some necessary condition for $x \in X(\A_K)$ to belong to $X(K)$, and his considerations explained why the Hasse principle failed in all the examples known at those times. He observed that one can pair any $\alpha \in \Br(X)$ with any $x = \{ x_v \} \in X(\A_K)$ by putting $$(\alpha, x)   \coloneqq \sum_v \mathrm{Inv}_v( \alpha_{|x_v}),$$ where $ \alpha_{|x_v}$ is the pullback of $\alpha$ to to $\Br(x_v) = \Br(K_v)$ via $x_v \hookrightarrow X$. 
This is always well defined, and it gives the pairing mentioned in the introduction 
$$X(\A_K) \times \Br(X) \rightarrow \Q / \Z. $$
It is then a direct consequence of \eqref{ctfBrauer} that 
$$ X(K) \subset X(\A_K)^{\Br(X)}   \coloneqq \{ x \in X(\A_K) \colon (\alpha, x) = 0 \,\, \forall \,\, \alpha \in \Br(X) \}, $$
and one says that there is a Brauer-Manin obstruction to the Hasse principle if $X(\A_K) \neq \emptyset$ but $X(\A_K)^{\Br(X)}  = \emptyset$. Theorem \ref{Tschinkel} in the introduction says that in order to compute $X(\A_K)^{\Br(X)} $ efficiently for a K3 surface $X/K$ one needs to bound the cardinality of $\Br(X)/ \Br_0(X)$. From the Leray spectral sequence 
$$E^{p,q}_2   \coloneqq \mathrm{H}^p \big( G_K, \mathrm{H}_{\et}^q(\overline{X}, \mathbb{G}_m ) \big) \Rightarrow \mathrm{H}_{\et}^{p+q}(X, \mathbb{G}_m),$$
one has the exact sequence
$$\Br(K) \rightarrow \ker \big(\Br(X) \rightarrow \Br(\overline{X})^{G_K} \big) \rightarrow \mathrm{H}^1 \big(G_K, \Pic(\overline{X}) \big) \rightarrow \mathrm{H}^3 ( G_K, \overbar{K}^\times ), $$
and the group $\mathrm{H}^3 ( G_K, \overbar{K}^\times )$ vanishes because $K$ is a number field. The filtration mentioned in the introduction 
\begin{equation}
 \Br_0(X) \subset \Br_1(X) \subset \Br(X) 
\end{equation}
is given by the constant classes $\Br_0(X)   \coloneqq \mathrm{Im} \big(\Br(K) \rightarrow \Br(X) \big)$ and the \textit{algebraic classes} $\Br_1(X)    \coloneqq \ker \big(\Br(X) \rightarrow \Br(\overline{X})^{G_K} \big)$. The quotient $\Br_1(X)/ \Br_0(X) \cong \mathrm{H}^1 (G_K, \Pic(\overline{X}) )$ is a finite group since  $\Pic(\overline{X}) \cong \NS(\overline{X})$ for K3 surfaces, and one can usually compute it after finding explicit generators of $\Pic(\overline{X})$. Finally, the \textit{transcendental Brauer group} is the quotient $\Br(X)/ \Br_1(X) \subset \Br(\overline{X})^{G_K}$ and its elements are of geometric nature, in the sense that they are represented by Azumaya algebras over $\overline{X}$. By \cite{MR2395136} the group $\Br(\overline{X})^{G_K}$ is always finite, and it turns out that the most fruitful strategy to bound the order $\Br(X)/ \Br_1(X)$ is to study $\Br(\overline{X})$ as a Galois representation. To do so, consider the Kummer sequence $$1 \rightarrow \mu_n \rightarrow \mathbb{G}_m \xrightarrow{n} \mathbb{G}_m \rightarrow 1 $$
after inspecting the long exact sequence associated in \et ale cohomology one obtains
$$ 0 \rightarrow \Pic(\overline{X}) \otimes \Z / n \Z \rightarrow \mathrm{H}^2_{\et}(\overline{X}, \mu_n) \rightarrow \Br(\overline{X}) [n] \rightarrow 0,  $$ 
which becomes 
\begin{equation} \label{eqnFINALE}
    0 \rightarrow \NS(\overline{X}) \otimes \Z / n \Z \rightarrow \mathrm{H}^2_{\et}(\overline{X}, \mu_n) \rightarrow \Br(\overline{X}) [n] \rightarrow 0,
\end{equation} 
since (in general, for any surface $X$) $\Pic(\overline{X})$ is an extension of $\NS(\overline{X})$ by the divisible group $\Pic^0(\overline{X})$. We have already introduced the following notation during the proof of Proposition \ref{invariance O}. We make it now official.
\begin{defi}
Let $X/K$ be a smooth projective surface defined over an algebraic closed field of characteristic $0$, and assume for simplicity that $  \mathrm{H}^{\bullet}_{\et}(\overline{X},\widehat{\Z})$ has no torsion. We denote by $\widehat{\NS}(\overline{X}) = \NS(\overline{X}) \otimes \widehat{\Z}$ the profinite completion of $\NS(\overline{X})$ and by $\widehat{T}(X)  \coloneqq \widehat{\NS}(\overline{X})^{\perp}$ the orthogonal complement of $\widehat{\NS}(\overline{X}) \subset \mathrm{H}^{2}_{\et}(\overline{X},\widehat{\Z})(1)$ (this is a non-standard notation). Note that any embedding $K \hookrightarrow \C$ induces an isomorphism $\widehat{T}(X) \cong T(X_\C) \otimes \widehat{\Z}.$
\end{defi}
After some manipulations of the sequence \ref{eqnFINALE}, and still assuming that $\mathrm{H}^{\bullet}_{\et}(\overline{X},\widehat{\Z})$ has no torsion, one finds that $$\Br(\overline{X}) \cong \big( \mathrm{H}^{2}_{\et}(\overline{X},\widehat{\Z})(1)/ \widehat{\NS}(\overline{X}) \big) \otimes \Q / \Z.$$
This isomorphism respects the Galois actions, thus describing the Galois module $\Br(\overline{X})$ in terms of the \et ale cohomology of $\overline{X}$. To make this even more explicit, note that since $H^2(X, \Z)$ is a unimodular lattice one has  
\begin{align*}
\mathrm{H}^{2}_{B}(X_{\C},\Z)(1)/\NS(X_{\C}) &\xrightarrow{\sim} \Hom(T(X_\C), \Z) \\
v + \NS(X_{\C}) &\mapsto (x \to (x,v)),
\end{align*}
so that $$\Br(\overline{X}) \cong \Hom(T(X_{\C}), \Q / \Z) \cong  \Hom(\widehat{T}(\overline{X}), \Q / \Z). $$
Note that from the equation above one gets a natural action of $\Oo(X)$ on $\Br(\overline{X})$.
A Hodge isometry $ f \colon T(X_\C) \xrightarrow{\simeq} T(Y_\C)$ naturally induces two maps on Brauer groups: $f^{*} \colon \Br(Y_\C) \rightarrow \Br(X_\C)$ given by applying the contravariant functor $\Hom(-. \Q/\Z)$ and $f_{*} \colon \Br(X_\C) \rightarrow \Br(Y_\C)$ given by identifying $$\Hom(T(X_\C),\Z) \cong \{v \in T(X)_\Q \colon (v,x)_X \in \Z \,\, \text{for all} \,\, x \in T(X)\}.$$ They are one the inverse of the other. Assume now that $X$ has CM by $\Oo_E$, we make the following elementary but useful definition. 
\begin{defi} \label{Level Structure T}
By a level structure on $T(X)$ we mean a finite subgroup $B \subset \Br(X)$ that is invariant under the action of $\Oo(X)$.
\end{defi}
 It is clear that level structures on $T(X)$ corresponds bijectively to free $\Z$-modules $\Lambda$  $$\Hom(T(X), \Z) \subset \Lambda \subset \Hom(T(X), \Q) $$ that are invariant under the action of $\Oo(X)$, or equivalently to ideals $I \subset \Oo_E$ by putting $$ \Br(X)[I]=\{ x \in \Br(X) \colon i \cdot x =0 \,\, \forall\,\, i \in I\}.$$
\begin{lemma}
Let $X/K \subset \C$ be a K3 surface defined over a number field $K$, and suppose that $X$ has CM over $K$. Then $\Br(\overline{X})^{G_{K}} \subset \Br(\overline{X})$ is a level structure on $T(X_\C)$. 
\end{lemma} 
 \begin{proof}
By the results in \cite{MR2395136}, we know that $\Br(\overline{X})^{G_{K}}$ is finite.  If $X/K \subset \C$ has CM over $K$, then $Br(X_{\overbar{K}})^{G_{K}} $ is also invariant under the $\Oo(X)$-action, since every cycle in $E(X)$ is defined over $K$. 
 \end{proof}
 
 \begin{rmk}
 We decided to define a level structure in this way because, ultimately, we will apply our results to study the Brauer group of $X$. The alternative is to define a level structure as 
 $$T(X)[I] \coloneqq \{ v \in T(X) \otimes \Q / \Z \colon i \cdot v = 0 \,\, \forall \,\, i \in I \},$$
and the results of this paper do not depend on this choice. 
  \end{rmk}

 \section{The main theorem of complex multiplication} \label{section main thm}
In his paper \cite{2005math......8018R}, Rizov proves an analogue of the main theorem of complex multiplication $K3$ surfaces. As a matter of fact, the theorem is a formal consequence of the fact (also proved by Rizov) that the moduli stack of polarized$K3$ surfaces over $\Q$ is related to the canonical model of the $K3$ Shimura variety via an \'{e}tale morphism defined over $ \Q$ (the period morphism). As pointed out by Madapusi Pera in \cite{MR3370622}, Rizov's theorem could also be proved using the theory of motives for absolute Hodge cycles, see \textit{loc. cit.} Corollary 4.4. In this section we follow the notations of Rizov's and of Milne's article `Introduction on Shimura varieties' (appearing in \cite{Arthur2005HarmonicAT}).

\subsection{A summary of class field theory}
Before stating the main theorem of complex multiplication, we recall the main statements from class field theory. Let $K$ be a number field. Class field theory describes $\Gal(K^{ab}/K)$ via the reciprocity map, which is a surjective, continuous morphism $$\rec_{K} \colon \A^{\times}_{K} \rightarrow \Gal(K^{ab}/K)$$
whose kernel contains $K^{\times}$. It induces an isomorphism $\widehat{K^\times \backslash \A^\times_{K}} \xrightarrow{\sim} \Gal(K^{ab}/K)$, where $\widehat{K^\times \backslash \A^\times_{K}}$ denotes the profinite completion of $K^\times \backslash \A^\times_{K}$. For our purposes, it is also useful to introduce the Artin map:
\begin{equation}
\art_K \colon \A_K^\times \xrightarrow{\rec_{K} } \Gal(K^{ab}/K) \xrightarrow{\sigma \mapsto \sigma^{-1}}   \Gal(K^{ab}/K).
\end{equation}

 The reciprocity map enjoys the following properties
\begin{enumerate}
\item If $L/K$ is an abelian extension, there is a commutative diagram 
\begin{center}

\begin{tikzcd} 
K^\times \backslash \A^\times_{K} \arrow[r, "\rec_K"] \arrow[d]
& \Gal(K^{ab}/K) \arrow[d, "\sigma \mapsto \sigma_{| L}"] \\
K^\times \backslash \A^\times_{K}/ \text{Nm}_{L/K}(\A^\times_{K} ) \arrow[r, "\sim"]
& \Gal(L/K).
\end{tikzcd}
\end{center}
This establishes a one-to-one correspondence between finite-index subgroups of $\A^\times_{K}$ that contain $K^\times$ and finite abelian extensions of $K$. 
\item If $v$ is a prime of $K$ that is unramified in $L$ and $\pi \in K_v$ is a prime element, then the id\`{e}le $( \cdots 1 \cdots \pi \cdots 1 \cdots )$ with $\pi$ at the $v$-component and $1$ elsewhere is sent by $\rec_K$ to the Frobenius element  $(v, L/K) \in \Gal(L/K)$.
\item If $K$ is totally imaginary, then the reciprocity map factors through the quotient $\A^{\times}_K \twoheadrightarrow \A^{\times}_{K,f}.$
\end{enumerate} 

Let now $V$ be a finite dimensional $\Q$-vector space and let $h \colon \DT \rightarrow \GL(V)_{\R}$ be a rational Hodge structure. Suppose that $h$ is special, i.e., that it satisfies the condition in Definition \ref{Special HS}. In particular, there exists a torus $T \subset \GL(V)$ defined over $\Q$ such that the morphism $h$ factors through $T_{\R}$: $$h \colon \DT \rightarrow T_{\R} \hookrightarrow \GL(V)_{\R}.$$
Recall that the reflex field $E(h)$ introduced in \eqref{Definition reflex field} is the field of definition of the composition $$\mathbb{G}_{m,\C} \xrightarrow{\mu} \mathbb{S}_\C \xrightarrow{h} T_\C,$$ so that we can consider the map $$h \circ \mu \colon  \mathbb{G}_{m,E(h)} \rightarrow T_{E(h)}.$$
By the functoriality of the Weil restriction of scalars, we also have a map $$\text{Res}_{E(h)/\Q}(h \circ \mu) \colon \text{Res}_{E(h)/\Q}(\mathbb{G}_{m,E(h)}) \rightarrow \text{Res}_{E(h)/\Q}(T_{E(h)}),$$
and we define the map $r'_h$ as the composition $$\text{Res}_{E(h)/\Q}(\mathbb{G}_{m,E(h)}) \rightarrow \text{Res}_{E(h)/\Q}(T_{E(h)}) \xrightarrow{\text{N}} T,$$
where $N$ is the \textit{Norm map}, acting on $\overline{\Q}$-points as \begin{align*}
  \text{Res}_{E(h)/\Q}(T_{E(h)})(\overline{\Q}) \cong \bigoplus_{\sigma \colon E(h) \hookrightarrow \overline{\Q}} T(\overline{\Q})_\sigma &\rightarrow T(\overline{\Q}) \\
  (t_\sigma)_\sigma &\mapsto \prod_{\sigma}t_\sigma.
\end{align*} 
Finally, we define $r_{h} \colon \A^{\times}_{E(h)} \rightarrow T(\A_{f})$ as the composition $$\A^{\times}_{E(h)} =  \Res_{E(h)/ \Q}(\mathbb{G}_{m,E(h)})(\A) \xrightarrow{r^{'}_{h}} T(\A) \xrightarrow{\text{proj}} T(\A_{f}).$$
In our case, where $T=U_E=\MT(X)$, we have 
\begin{prop}
After naturally identifying $\MT(X)$ with the norm-1 torus $\U_E \subset \Res_{E/\Q} \mathbb{G}_m$ the map $r$ corresponds to \begin{align*}
 r \colon \A^{\times}_{E} &\mapsto \A^{\times}_{E,f} \\
 s &\to \frac{s_f}{\bar{s}_f}
\end{align*} 
\end{prop}

\begin{proof}
Remember that the reflex field $E$ is naturally embedded into $\C$, via the evaluation map. Denote by $\tilde{E} \subset \C$ its Galois closure, and consider the natural embedding
\begin{align*}
E &\hookrightarrow E \otimes_{\Q} \tilde{E} \\
e &\to e \otimes 1 .
\end{align*}
We can multiply every element $x \in E \otimes_{\Q} \tilde{E}$ by an element of $e \in E$ in two ways, respectively $e \cdot x$ and $x \cdot e$.
Denote by $\mathcal{E}   \coloneqq \{  \iota \colon E \hookrightarrow \tilde{E} \}$ the set of embeddings. The Galois group  $G   \coloneqq \Gal(\tilde{E} / \Q) $ acts transitively on $\mathcal{E}  $ by $\iota \mapsto g \iota$. There is a decomposition $$ E \otimes_{\Q} \tilde{E} = \bigoplus_{\iota \in \mathcal{E}} \tilde{E}_{\iota} $$
where 
$$ \tilde{E}_{\iota} = \{ x \in E \otimes_{\Q} \tilde{E} \colon e \cdot x = x \cdot \iota(e) \,\, \forall \,\, e\in E \}.$$
One can show that exists a unique element $1_{\iota} \in \tilde{E}_{\iota}$ such that the map $\tilde{E} \rightarrow \tilde{E}_{\iota}$, $\tilde{e} \mapsto 1_\iota \cdot \tilde{e}$ is an isomorphism of fields (multiplication on $\tilde{E}_{\iota}$ being the one induced by $E \otimes_{\Q} \tilde{E}$).
If we let $G$ act on the right side, i.e. $g( z \otimes w)  \coloneqq z \otimes g(w)$ for every $g \in G$, we have $g( 1_{\iota} \cdot \tilde{e}) = 1_{g \iota} \cdot g(\tilde{e}) $. In particular, the natural embedding $E \hookrightarrow E \otimes_{\Q} \tilde{E}$ becomes 
\begin{align} \label{alignimportante}
E &\hookrightarrow \bigoplus_{\iota \in \mathcal{E}} \tilde{E}_{\iota} \\
e &\mapsto \oplus_\iota 1_\iota \cdot \iota(e).
\end{align}
In our case, denote by $\sigma \colon E \hookrightarrow \tilde{E}$ the canonical inclusion. The cocharacter is given by 
\begin{align*}
\mu \colon E &\rightarrow E \otimes E \subset E \otimes \tilde{E} \\
e &\mapsto (1_{\sigma} \cdot \sigma(e), \cdots, \cdot 1_{ \overline{ \sigma}} \cdot \sigma(e)^{-1}), 
\end{align*}
where all the dotted entries are $1$. 
Denote by $S \subset G$ the stabilizer of $\iota$, the map $r'$ is finally  given by 
$$ \prod_{[g] \in G/ S} [g] \mu(e) =  \sum_{ \iota \in \mathcal{E}} 1_\iota \cdot  \iota \bigg( \frac{e}{\bar{e}} \bigg) = \frac{e}{\bar{e}},$$ 
(note that $[g] \sigma(e)$ is well defined) where in the last equality we use the identification \eqref{alignimportante}.
\end{proof}
We can now state the main theorem of CM for $K3$ surfaces: 
\begin{thm}[Rizov] \label{Rizov}
Let $X/\C$ be a $K3$ surface with complex multiplication and let $E \subset \C$ be its reflex field. Let $\tau \in \Aut(\C / E)$ and $s \in \A^{\times}_{E,f}$ be a finite id\`{e}le such that $\art(s) = \tau_{|E^{ab}}$. There exists a unique Hodge isometry $\eta \colon T(X)_{\Q} \rightarrow T(X^{\tau})_{\Q}$ such that the following triangle commutes 

\begin{center}
\begin{tikzcd}[row sep= large, column sep =large]
\widehat{T}(X)_{\Q} \arrow[r, "\eta \otimes \A_{f}"] & \widehat{T}(X^\tau)_{\Q} \\
\widehat{T}(X)_{\Q} \arrow[u, "\frac{s}{\bar{s}}"]\arrow[ur, "\tau^*"]
\end{tikzcd}
\end{center}
where $\tau^*$ is the pull-back in \'{e}tale cohomology of $\tau \colon X^{\tau} \to X$. 
\end{thm}

\begin{proof}
The diagram above, as found in \cite{2005math......8018R}, reads a bit differently: 

\begin{center}
\adjustbox{scale=1.1}{
\begin{tikzcd}
P_B(X,\A_{f})(1) \arrow[r, "\tilde{\eta} \otimes \A_{f}"] & P_B(X^{\tau},\A_{f})(1) \\
P_B(X, \A_{f})(1)\arrow[u, "r_{X}(s)"]\arrow[ur, "\tau^*"],
\end{tikzcd}
}
\end{center}
where $P_B(X,\A_{f})(1) $ is the primitive cohomology of $X$ with respect to some polarization $\ell \in \NS(X)$, $\tilde{\eta} \colon P_B(X,\Q)(1)\rightarrow P_B(X^{\tau},\Q(1)(1)$ is a Hodge isometry and $r_{X}$ is the reciprocity map associated to the torus $\MT(P_B(X,\Q)(1))$. Now, $P_B(X,\Q)(1) = T(X)_{\Q} \oplus A$, where $A$ is the rational $(0,0)$-part of $P_B(X,\Q)(1)$, i.e. $A= \{v \in \NS(X)_{\Q} \colon (v,\ell) = 0\}$. It is therefore clear that the inclusion $ T(X)_{\Q} \hookrightarrow P_B(X,\Q)(1)$ induces an isomorphism of Mumford-Tate groups 
\begin{center}

\begin{tikzcd}
\GL(T(X)_{\Q}) \arrow[r, hookrightarrow] &\GL( P_B(X ,\Q)(1)) \\
\MT(T(X)_{\Q})\arrow[u, hookrightarrow]\arrow[r, "\cong"] & \MT(P_B(X,\Q)(1)) \arrow[u, hookrightarrow].
\end{tikzcd}
\end{center}
This identification implies $r_{X}(s) = \frac{s}{\bar{s}}$ and $\tilde{\eta} = (\eta , \tau^{*})$, where $\tau^{*} \colon \NS(X) \rightarrow \NS(X^{\tau})$ is the pull-back via $\tau$. 
\end{proof}

\section{Ideal lattices and id\`eles} \label{Section6}
Ideal lattices provide a natural way to classify the transcendental lattices of K3 surfaces with CM by a fixed $\Oo_E$. All the results in here are classical, and we mainly follow Chapter 6 of \cite{LangCM}, Chapter 6 of \cite{shimura2014arithmetic}, and \cite{bayer-fluckiger_2002}.
\begin{defi}
Let $E$ be a CM number field. An ideal lattice $(I,q)$ is a fractional ideal $I \subset E$ and a non-degenerate, symmetric, $\Q$-bilinear form$$q \colon I \times I \rightarrow \Q$$ such that $q(\lambda x, y) = q(x, \overline{\lambda} y)$ for every $x,y \in I$ and $\lambda \in \mathcal{O}_{E}$. 
\end{defi}
By the non-degeneracy of the trace, it follows that there exists $\alpha \in E$ such that $\alpha = \overline{\alpha}$ and $q(x,y) = \tr_{E/ \Q}(\alpha x \overline{y})$. So that, from now on, we will denote with $(I, \alpha)$ the ideal lattice $(I,q)$ with $q(x,y) = \tr_{E/ \Q}(\alpha x \overline{y})$. 
\begin{defi}
An ideal lattice $(I, \alpha)$ is said to be \textit{integral} if $q$ takes value in $\Z$, and \textit{even} if $q(x,x) \in 2 \Z$ for every $x \in I$.
\end{defi}
 Recall that the inverse different ideal $\mathcal{D}_{E}^{-1}$ is defined to be the maximal fractional ideal of $E$ where $\tr_{E/ \Q}$ takes integral values. Hence, if $\alpha \in E$ is as above, $(I, \alpha)$ is integral if and only if 
\begin{equation} \label{integrality}
(\alpha) I \bar{I} \subset \mathcal{D}_{E}^{-1}.
\end{equation}
Let $(I, q)$ be an integral ideal lattice. Its dual is defined as $(I^{\vee},q)$ where 
\begin{equation} \label{dual lattice definition}
I^{\vee} = \lbrace x\in E \, : \, q(x,I) \subset \Z \rbrace.
\end{equation}
Note that the quadratic form induces a natural isomorphism $I^\vee \xrightarrow{\sim} \Hom(I,\Z)$ given by $x \mapsto q(x,-)$.
We also have a natural inclusion  $(I, q) \subset (I^{\vee},q)$. From the definition, it follows that also $(I^{\vee},q)$ is an ideal lattice (usually non integral) and that $$I^{\vee} = \big( \alpha\bar{I} \mathcal{D}_{E} \big)^ {-1};$$
the inclusion $I \subset \big( \alpha\bar{I} \mathcal{D}_{E} \big)^ {-1}$ is hence also a consequence of \eqref{integrality}. \\
\begin{defi}
We say that two ideal lattices $(I, \alpha)$ and $(J, \beta)$ are equivalent, $(I, \alpha) \cong (J, \beta)$, if there exists $e \in E^{\times}$ such that $J = eI$ and $\alpha = e \bar{e}\beta$. 
\end{defi}
This means exactly that multiplication by $e$ $$e \colon I \rightarrow J$$ is an isometry. Note that the two lattices $(I, \alpha)$ and $(J, \beta)$ can be isometric without being equivalent (because a general isometry between the two might not be $E$-linear). 
\begin{rmk}\leavevmode
Note that if $(I, \alpha) \cong (J, \beta)$ via $e \in E^{\times}$, then $(I^{\vee}, \alpha) \cong (J^{\vee}, \beta)$ via $e$ as well  Lemma \ref{duals}. 
\end{rmk}
If $(I, \alpha)$ is an ideal lattice, the quotient $E / I \cong I \otimes \Q / \Z$ is a torsion abelian group, and also an $\mathcal{O}_{E}-$module. We now make the analogue of Definition \ref{Level Structure T}.
\begin{defi} \label{Level structure II}
By a level structure on the ideal lattice $(I,\alpha)$ we mean a finite, $\Oo_E$-invariant subgroup $G \subset I^\vee \otimes \Q / \Z$. 
\end{defi}
\begin{rmk} \label{Important remark}
To give a level structure is equivalent to give a fractional ideal $J$ such that $I^\vee \subset J$, i.e. $J = \pi^{-1}(G)$ where $\pi \colon E \rightarrow E / I^\vee$ is the canonical projection. Equivalently, this is as giving the ideal $I_G = I^\vee J^{-1} \subset \Oo_E$, and from now on we will not make any distinction between one or the other definitions. 
\end{rmk}
We want now to extend the definition of equivalence keeping track of level structures. So let $(I, \alpha, G)$ and $(J, \beta, H)$ be two ideal lattices with level structures. We say that $(I, \alpha, G) \cong (J, \beta, H)$ if there exists $e \in E^{\times}$ as before such that the map induced by multiplication by $e$ $$ E / I^{\vee} \rightarrow E / J^{\vee} $$ restricts to an isomorphism between $G$ and $H$. In other words, there is an action of $E^{\times}$ on the set of ideal lattices with a level structure by putting for any $e \in E^\times$ $$e \cdot (I, \alpha, G)   \coloneqq \Big(eI, \frac{\alpha}{e\bar{e}}, eG \Big)$$ 
where $eG$ is the image of $G$ under the map $$e \colon E/ I^{\vee} \rightarrow E/ eI^{\vee},$$
where the last equation makes sense since $eI^{\vee} = (eI)^{\vee}$ thanks to the above remark. 
The following local facts can be found in Lang \cite{LangCM} Chapter 6. 
\begin{prop}
Let $I,J \subset E $ be fractional ideals. Then:
\begin{enumerate}
\item For all but finitely many finite places $v$ of $E$, $I \otimes \mathcal{O}_{E,v} = J \otimes \mathcal{O}_{E,v},$
\item  $I \subset J$ if and only if $I \otimes \mathcal{O}_{E,v} \subset J \otimes \mathcal{O}_{E,v}$ for every finite place $v$,
\item If $(I_{v})_{v}$ is a collection of $\mathcal{O}_{E,v}-$modules $I_{v} \subset E_{v}$, such that for all but finitely many $v$'s we have that $I_{v} = \mathcal{O}_{E,v}$, than there exists unique a fractional ideal $I$ such that $I \otimes \mathcal{O}_{E,v} = I_{v}$ for every $v$.

\end{enumerate}
\end{prop}
Let now $s \in \mathbb{A}^{\times}_{E,f}$ be a finite id\`ele and $I$ a fractional ideal. Then there exists a unique fractional ideal $J$ such that $$J_{v} = (s_{v}) \cdot I_{v}$$ since for all but finitely many $v$'s we have $(s_{v}) \cdot I_{v} = I_{v}$. The ideal $J$ corresponds to $(s)I$, where we denote
 $$(s) = \prod_{\p} \p^{\ord_{\p}(s)}$$
 (for simplicity, we shall also denote $(s) \cdot J$ by $sJ$).  To extend the action of $E^{\times}$ on triples $(I, \alpha, G)$ to a subgroup of $\A^{\times}_{E,f}$ containing $E^{\times}$, one starts from the isomorphism (pag. 77 in Lang's book \cite{LangCM})  $$E / I \cong \bigoplus E_{\p} / I_{\p} $$
where the sum is taken over all the prime ideals of $\mathcal{O}_{E}$, and the natural homorphism $$\A_{E,f} \rightarrow E/I$$ whose kernel is exactly $\oplus I_{\p}$. If $s \in \A^{\times}_{E,f}$ is an id\`ele, we saw before that $J   \coloneqq sI$ is the only fractional ideal of $E$ such that $J_{\p} = s_{\p} I_{\p}$. Hence, we obtain a commutative square  
\begin{center}
\begin{tikzcd} \label{mult by idele}
\A_{E,f} \arrow[r] \arrow[d, "s"]
& E/I \arrow[d, "\psi"] \\
\A_{E,f} \arrow[r]
& E/sI
\end{tikzcd}
\end{center}
where $\psi$ is given at the $\p$-component by multiplication by $s_{\p}$.
If $G \subset E/I$ is a subgroup, we denote by $sG \subset E/ sI$ the image of $G$ under $\psi$ in the diagram above. 
In order to extend the action of $E^{\times}$, we make the following definition. 
\begin{defi}
Let $F \subset E$ be the fixed field of the complex conjugation, we define $K_{E} \subset \A^{\times}_{E,f} $ to be the kernel of $$ \A^{\times}_{E,f} \xrightarrow{\text{Nm}                                          _{E/F}} \A^{\times}_{F,f} \twoheadrightarrow C_{F} $$ where $C_{F}$ is the id\`ele class group of $F$. Equivalently, $s \in K_{E}$ if and only if $s \bar{s} \in F^{\times}$
\end{defi}
Let now $(I, \alpha)$ be an ideal lattice and $s \in K_E$. Define $$ s \cdot (I, \alpha)   \coloneqq \bigg( sI, \frac{\alpha}{s \bar{s}} \bigg).$$
If $(I, \alpha)$ is integral, then also $s \cdot (I, \alpha)$ is integral. We show that this construction commutes with formation of duals. 
\begin{lemma} \label{duals}
Let $(I, \alpha)$ be an ideal lattice, and let $s \in K_E $. Then the dual of $ s \cdot (I, \alpha)$ is $ s \cdot (I, \alpha)^{\vee}$.
\end{lemma}

\begin{proof}
Indeed, the dual of  $ s \cdot (I, \alpha)$ is $$ \bigg( (s \bar{s}) (\alpha^{-1}) \mathcal{D}^{-1}_{E} (\bar{s}^{-1}) \bar{I}^{-1}, \frac{\alpha}{s \bar{s}} \bigg) = \bigg( (s) (\alpha^{-1}) \mathcal{D}^{-1}_{E}  \bar{I}^{-1}, \frac{\alpha}{s \bar{s}} \bigg)$$
and $$ s \cdot (I, \alpha)^{\vee} = s \cdot \bigg(  \big( \alpha\bar{I} \mathcal{D}_{E} \big)^ {-1}, \alpha \bigg) =  \bigg( (s) (\alpha^{-1}) \mathcal{D}^{-1}_{E}  \bar{I}^{-1}, \frac{\alpha}{s \bar{s}} \bigg).$$
\end{proof}
This commutativity allows us to make the following definition.
\begin{defi}
Let $(I, \alpha, G)$ be an ideal lattice with level structure, and let $s \in K_E$. Then we define 
$$ s \cdot (I, \alpha, G)   \coloneqq \bigg( sI, \frac{\alpha}{s \bar{s}}, sG \bigg),$$
where $sG$ is the image of $G$ under multiplication by $s$ $$ E/I^{\vee} \rightarrow E/sI^{\vee} =  E/(sI)^{\vee}.$$ 
\end{defi}

\subsection{Generalization of Proposition \ref{Infinite pic 20}} 
In this subsection we prove an analogue of Proposition \ref{Infinite pic 20} for any CM number field $E$ with $[E \colon \Q] \leq 10$. 

\begin{defi} \label{definition k3 type}
A Hodge structure of weight two on $\Lambda$ (the K3 lattice) with $\dim_\C \Lambda^{2,0} =1$ and such that any $\omega \in \Lambda^{2,0}-0$ satisfies 
\begin{itemize}
\item $(\omega, \omega) = 0$;
\item $(\omega, \overline{\omega}) > 0$;
\end{itemize}
is said to be of \textit{K3 type}.
\end{defi}
\begin{rmk}
This implies that the whole Hodge structure can be reconstructed by $\omega$, since $\Lambda^{1,1}$ corresponds to the complexification of $\langle \text{Re}(\omega), \text{Im}(\omega) \rangle^{\perp} \subset \Lambda_\R$, see Chapter 6, Proposition 1.2 of \cite{MR3586372}. 
\end{rmk}
The following theorem is the surjectivity of the period map (see Chapter 6, Remark 3.3. of \cite{MR3586372}). 
\begin{thm}
Let us consider $\Lambda$ endowed with a Hodge structure of K3 type. Then there exists a complex K3 surface $X$ with a Hodge isometry $\Lambda \cong \mathrm{H}^2(X,\Z)$.
\end{thm}
\begin{prop} \label{key proposition} 
Let $E$ be a CM number field with $[E \colon \Q] \leq 10$. Then there are infinitely many $\C$-isomorphism classes of principal K3 surfaces with CM by $E$. 
\end{prop}

\begin{proof}
This is a consequence of Corollary 1.12.3 in Nikulin's paper \cite{MR525944}. Let $S$ be an even lattice of signature $(s_{(+)}, s_{(-)})$ and let $\Lambda$ be an even unimodular lattice of signature $(\lambda_{(+)},\lambda_{(-)})$. Nikulin's result says that a primitive embedding $S \hookrightarrow \Lambda$ exists if the following conditions are satisfied:
\begin{enumerate}
\item $\lambda_{(-)} - s_{(-)} \geq 0$ and $\lambda_{(+)} - s_{(+)} \geq 0$;
\item Let $g$ be the minimum number of generators of $S^\vee/S$. Then $\lambda_{(+)} + \lambda_{(-)} -  s_{(+)}  - s_{(-)} > g.$ 
\end{enumerate}
Note that, in our case, where the lattices are non-degenerate, $\lambda_{(+)} + \lambda_{(-)} = \text{rank}(\Lambda)$ and $s_{(+)}  + s_{(-)} = \text{rank}(S)$. Moreover, $g \leq \text{rank}(S)$ always. In particular, a primitive embedding exists every time that \begin{equation} \label{rank condition Nikulin}
\text{rank}(\Lambda) > 2 \cdot \text{rank}(S)
\end{equation}
Let us now prove the proposition. Write $[E \colon \Q] =2n$ with $n \leq 5$. Consider $\alpha \in F^{\times}$ with $\alpha \mathcal{D}_{E} \subset \Oo_E$, so that the ideal lattice $(\Oo_E, \alpha)$ is an even integral lattice, and assume that only one embedding $\sigma' \colon F \hookrightarrow \C$ satisfies $\sigma'(\alpha) > 0$ (note that one can always find such an element of $F$). Let us denote by $\sigma \colon E \hookrightarrow \C$ an extension of $\sigma'$ (the other extension will be given by $\bar{\sigma}$). This choice of $\alpha$ ensures that the signature of $(\Oo_E, \alpha)$ is $(2, 2n-2)$. We would like to produce an algebraic K3 surface using the surjectivity of the period map. Since $[E \colon \Q] \leq 10$ and by the choice of $\alpha$, we readily check that conditions $(1),(2),(3)$ above are satisfied for $(\Oo_E, \alpha)$ and $\Lambda$, so we can find a primitive embedding of lattices $(\Oo_E, \alpha) \subset \Lambda$. We want now to endow $\Lambda$ with a Hodge structure which corresponds to a K3 surface with CM by $\Oo_E$. To do so, consider again the decomposition $$\Oo_E \otimes \C = \bigoplus_{\tau \colon E \rightarrow \C} \C_\tau$$ and put $\Lambda^{2,0}   \coloneqq \C_\sigma$, where we consider $\Oo_E \otimes \C \subset \Lambda \otimes \C$. Let us call $X$ the corresponding K3 surface. It is straightforward to show that $T(X) = (\Oo_E, \alpha)$, and that $\End_{\Hdg}(T(X)) = \Oo_E$. To show that this K3 surface is algebraic it is sufficient to find a class $L \in \NS(X)$ with $L^2 > 0$ by Theorem IV.6.2 of \cite{Barth}. But this class must exists because the signature of $\NS(X)$ is $(1, 21-2n)$, so that $X$ is algebraic. Finally, note that we can produce infinitely many $\alpha$'s such that the ideal lattices $(\Oo_E, \alpha)$ are pairwise non-isomorphic, so that we obtain infinitely many $\C$-isomorphism classes of K3 surfaces with CM by $\Oo_E$. 
\end{proof}
\begin{rmks}

\begin{enumerate}
\item This latter consideration shows perhaps the biggest different between the theory of CM K3 surfaces and the theory of CM abelian varieties. If $(I, \alpha)$ is an ideal lattice as above, then for every totally positive $f \in \Oo_E$ also $(I,f \alpha)$ corresponds to the transcendental lattice of a K3 surface. Moreover, if $f$ is not a unit, these two surfaces cannot be isomorphic. This allows one to obtain infinitely many different K3 surfaces just by changing the polarization. On the contrary, the polarization does not play a role in the classification of abelian varieties with CM by $\Oo_E$, and in this case the only invariants are the \textit{type} and the class of $I$ in $\text{Cl}(E)$. 
\item Note that the surfaces in the proposition are uniquely determined by $T(X)$. In fact a refinement of Nikulin result employed before (Theorem 1.14.4 of \cite{MR525944}) ensures that if $[E \colon \Q] \leq 10$ then the primitive embeddings $T(X) \hookrightarrow \Lambda$ form a single orbit under the natural action of the isometries of $\Lambda$. This implies that any Hodge isometry $ T(X^\sigma) \xrightarrow{\sim} T(X)$ extends to a Hodge isometry $f \colon H^2_B(X^\sigma, \Z) \xrightarrow{\sim} H^2_B(X, \Z)$, and even if $f$ might not be induced by an isomorphism, the K3 surfaces $X^\sigma$ and $X$ are nevertheless isomorphic because of the global Torelli theorem. 
\item If $[E \colon \Q ]\geq 12$, there are infinitely isomorphism classes of complex K3 surfaces with CM by $E$ by the result of Taelman mentioned in the introduction. Of course, the problem here is that these K3 surfaces might not be principal. One can still produce for some CM number field of higher degree an ideal lattice  whose discriminant group has small length, so to apply Nikulin result, but there is no general way to do it. 
\end{enumerate} 
\end{rmks}

\section{Type of a principal K3 surface with CM} \label{Section7}
In this section we introduce the \textit{type} of a K3 surface with CM. We saw in Proposition \ref{key proposition} how to construct Hodge structures of K3 type with CM starting from an integral ideal lattice. One starts with a CM number field $E$ and an embedding $\sigma \colon E \hookrightarrow \C$, and considers an even ideal lattice with level structure $(I, \alpha, G)$, with $\alpha \in F^\times$ such that $(\alpha)\D_E \subset \Oo_E$ and only $\sigma, \bar{\sigma} \colon E \hookrightarrow \C$ satisfies $\sigma(\alpha) >0$. To this data one associates a polarized Hodge structure of weight zero together with a level structure that we will denote by $(I, \alpha, G, \sigma)$, with $I^{1,-1} = \C_\sigma$. 
Let now $(X, B, \iota)$ be a principal CM $K3$ surface $X / \C$ with level structure $B \subset \Br(X)$ and an isomorphism $\iota \colon E \rightarrow E(X)$. Via the map $\iota$, we consider $T(X)$ an $\Oo_E$-module. 
\begin{defi} \label{TYPE}

We say that $(T(X), B, \iota)$ is of type $(I, \alpha, G, \sigma)$ if there exists an isomorphism of $\Oo_E-$modules $$\Phi \colon T(X) \xrightarrow{\sim} I $$ such that: 
\begin{enumerate}
\item $(v,w)_{X} = \tr_{E/ \Q} \Big( \alpha \Phi(v) \overline{\Phi(w)} \Big)$ for every $v,w \in T(X)$;
\item If $\Phi^{\vee} \colon T(X)^{\vee}  \rightarrow  I^{\vee} $ is the induced map on dual lattices, then $$\Phi^{\vee} \otimes \Q / \Z \colon  E/ I^{\vee} \rightarrow \Br(X) $$ sends $G$ isomorphically to $B$;
\item $\sigma_X \circ \iota = \sigma. $
\end{enumerate}

\end{defi}

\begin{rmks}\leavevmode
\begin{enumerate}
\item Here, with $\Phi^\vee$ we mean the induced map $$T(X)^\vee = \{v \in T(X)_\Q \colon (v,x) \in \Z \,\, \text{for all} \,\, x \in T(X) \} \rightarrow I^\vee,$$ where $I^\vee$ was defined in \eqref{dual lattice definition}. 
\item It may seem that fixing an abstract field $E$ together with the maps $\iota$ and $\sigma$ is redundant, since to every K3 surface $X / \C$ with CM one has canonically associated its reflex field $E$ (already in $\C$) together with an isomorphism $\sigma_X \colon E(X) \rightarrow E$. Fixing an abstract field $E$ allows us to keep track of the $\Aut(\C)$-action on $E(X)$: if $\tau \in \Aut(\C)$, we put $(T(X), B, \iota)^{\tau} = (T(X^{\tau}), \tau_*B,\tau^{ad} \circ  \iota).$ See Lemma \ref{lemma:nuovo}. 
\item Every CM $K3$ surface has a type: let $E \xrightarrow{\sigma} \C$ be its reflex field, put $\iota   \coloneqq \sigma_X^{-1}$ and choose $0 \neq v \in T(X) $.  The inverse image of $T(X)$ under the isomorphism $E \rightarrow T(X)_{\Q},$ $e \mapsto \iota(e) \cdot v$ is a lattice in $E$ invariant under the action of $\mathcal{O}_{E}$, hence it is a fractional ideal. By the non-degeneracy of the trace, we can find unique $\alpha \in E$ as in Definition \ref{TYPE}.
\end{enumerate}
\end{rmks}
\begin{defi} \label{iso k3 tuples}
Let $X, Y / \C$ be two principal K3 surfaces with CM. 
We say that the two triples $(T(X), B, \iota_X)$ and $(T(Y), C, \iota_Y)$ are isomorphic if there exists a Hodge isometry $f \colon T(X) \xrightarrow{\simeq} T(Y)$ such that 
\begin{enumerate}
\item $f^{ad} \circ \iota_X = \iota_Y$, where $f^{ad} \colon E(X) \rightarrow E(Y)$ is the induced isomorphism and
\item $f_{*} \colon \Br(X) \rightarrow \Br(Y)$ restricts to an isomorphism between $B$ and $C$, where $f_{*}$ is the induced map on Brauer groups (introduced in the discussion before Definition \ref{Level Structure T}). 
\end{enumerate}

\end{defi}
The following is an instance of point (2) in the remark above. 
\begin{lemma} \label{lemma:nuovo}
Let $X / \C$ be a principal K3 surface with CM and let $\iota \colon E \rightarrow E(X)$ be an isomorphism. Let $\tau \in \Aut(\C)$, and suppose that $(T(X), \iota) \cong (T(X^{\tau}), \tau^{ad} \circ \iota)$. Then $\tau$ fixes the reflex field of $X$. 
\end{lemma}
\begin{proof}
Note that $\sigma_X = \sigma_{X^{\tau}} \circ f^{ad}$ since $f$ is a Hodge isometry. During the proof of Proposition \ref{reflex field and CM}, we also proved that $\sigma_{X^{\tau}} \circ \tau^{ad} = \tau \circ \sigma_X$. By assumption, $f^{ad} \circ \iota = \tau^{ad} \circ \iota,$ i.e. $f^{ad} = \tau^{ad}$. Hence, $\sigma_X = \tau \circ \sigma_X$, i.e. $\tau$ fixes the reflex field of $X$. (On the other hand, if $X$ can be defined over $\Q$, then $T(X) \cong T(X^{\tau})$ for every $\tau \in \Aut(\C)$).
\end{proof}

\begin{lemma} \label{Uniq}
Suppose that $(T(X), B, \iota)$ is of type $(I, \alpha, G, \sigma)$ and let $\Phi$ and $\Phi'$ be two maps as in Definition \ref{TYPE}. Then there exists a root of unity $\mu \in \Oo_E^{\times}$ such that $\Phi = \mu \Phi'$. 
\end{lemma}

\begin{proof}
Indeed, the map $\Phi' \circ \Phi^{-1} \colon (I , \alpha) \rightarrow (I, \alpha)$ is an isometry, hence a root of unity. 
\end{proof}

We are ready to prove the following proposition. 
\begin{prop} \label{type thm}
Let $(T(X), B, \iota_X)$ be of type $(I, \alpha, G, \sigma)$ and let $(T(Y), C, \iota_Y)$ be of type  $(J, \beta, H, \theta)$. Then $(T(X), B, \iota_X) \cong (T(Y), C, \iota_Y)$ if and only if $(I, \alpha, G) \cong (J, \beta, H)$ and $\sigma = \theta$. 
\end{prop}

\begin{proof}
Let us prove the implication $(T(X), B, \iota_X) \cong (T(Y), C, \iota_Y) \Rightarrow (I, \alpha, G) \cong (J, \beta, H)$ and $\sigma = \theta$. 
Consider the square 
\begin{center}
\begin{tikzcd} 
T(X) \arrow[r, "\Phi_{X}"] \arrow[d, "f"]
& I\arrow[d, dashrightarrow] \\
T(Y) \arrow[r, "\Phi_{Y}"]
& J,
\end{tikzcd}
\end{center}
where $f$ is a map as in Definition $\eqref{iso k3 tuples}$ and $\Phi_X, \Phi_Y$ are the maps realizing the types of $X$ and $Y$ respectively. By linearity, we see that the dashed arrow is induced by multiplication by some $e \in E^{\times}$, which is also an isometry between the two ideal lattices $(I, \alpha)$ and $(J, \beta)$, i.e. $eI = J$ and $e \bar{e} \beta = \alpha $. The induced square on Brauer groups is  
\begin{center}
\begin{tikzcd} 
\Br(X) \arrow[r, "\Phi_{X}^{\vee}"] \arrow[d, "f_{*}"]
& E/ I^{\vee} \arrow[d, "e"] \\
\Br(Y) \arrow[r, "\Phi_{Y}{\vee}"]
& E / J^{\vee},
\end{tikzcd}
\end{center}
which implies $eG = H$, since $f_*(B) = C$, $\Phi_{X}^{*}(B) = G$ and  $\Phi_{Y}^{\vee}(C) = H$. By the definition of type we see that $\sigma_X \circ \iota_X = \sigma$ and $\sigma_Y \circ \iota_Y = \theta$. Moreover $f^{ad} \circ \iota_X = \iota_Y$ (by Definition \ref{iso k3 tuples}) and $\sigma_X = \sigma_Y \circ f^{ad}$ (since $f$ is a Hodge isometry). Hence, we see that $\sigma = \theta$. On the other hand, suppose that $(I, \alpha, G) \cong (J, \beta, H)$ and that $\sigma = \theta$, and let $e \in E^{\times}$ be an element realizing the equivalence. Consider the diagram 
\begin{center}
\begin{tikzcd} 
T(X) \arrow[r, "\Phi_{X}"] \arrow[d, dashrightarrow]
& I\arrow[d, "e"] \\
T(Y) \arrow[r, "\Phi_{Y}"]
& J,
\end{tikzcd}
\end{center}
and call $f$ the dashed arrow. Then, $f$ is an isometry between the lattices $T(X)$ and $T(Y)$ and satisfies condition 2 in Definition \ref{iso k3 tuples}. We need to prove that it respects the Hodge decomposition and that $f^{ad} \circ \iota_X = \iota_Y$. Since $\sigma = \theta$, we see that $\sigma_X \circ \iota_X = \sigma_Y \circ \iota_Y$. Let $0 \neq \omega \in T^{1,-1}(X)$ be a non-zero two form, and let $x \in E$. We want to show that $$\iota_Y(x) \cdot f(\omega) =\sigma_Y(\iota_Y(x)) f(\omega).$$ We compute 
$$\iota_Y(x) \cdot f(\omega) = \iota_Y(x) \cdot \Phi_Y^{-1} \big(e \Phi_X(\omega) \big) = \Phi_Y^{-1} \big(xe \Phi_X(\omega) \big) = \Phi_Y^{-1} \big(e \Phi_X(\iota_X(x) \cdot \omega) \big) = $$ $$  = f(\iota_X(x) \cdot \omega) = f( \sigma_X(\iota_X(x)) \omega) = \sigma_X(\iota_X(x)) f(\omega) = \sigma_Y (\iota_Y(x))f(\omega).$$
Hence, $f$ respects the Hodge decomposition. As a consequence of this, we must also have that $\sigma_X = \sigma_Y \circ f^{ad}$. Pre-composing with $\iota_X$ and using again the fact that $\sigma_X \circ \iota_X = \sigma_Y \circ \iota_Y$, we conclude. 
\end{proof}

\section{Main theorem of CM for K3 surfaces (after Shimura)} \label{Section8}

The next step is to translate Theorem \ref{Rizov} in the language of ideal lattices. 
\begin{thm} \label{MTSHIMURA}
Let $X/\C$ be a principal $K3$ surface with complex multiplication and reflex field $E \subset \C$. Let $\tau \in \Aut(\C / E)$ and let $s \in \A^{\times}_{E,f}$ be a finite id\`ele such that $\art(s) = \tau_{|E^{ab}}$. Suppose that $(T(X), B, \iota)$ is of type $(I,\alpha, G, \sigma)$. Then $(T(X^{\tau}), \tau_{*}B, \tau^{ad} \circ \iota )$ is of type $$\frac{s}{\bar{s}}  \cdot \big( I, \alpha, G, \sigma \big).$$
Moreover if $\Phi_{X}$ is a map realizing the \textit{type} of $X$, there exists a unique map $\Phi_{X^{\tau}}$ realizing the above \textit{type} of $X^{\tau}$, such that the following commutes
\begin{center}
\begin{tikzcd} 
\Br(X) \arrow[r, "\Phi_{X}^{\vee}"] \arrow[d, "\tau_{*}"]
& E/ I^{\vee} \arrow[d, "\frac{s}{\bar{s}}"] \\
\Br(X^{\tau}) \arrow[r, "\Phi_{X^{\tau}}^{\vee}"]
& E / \frac{s}{\bar{s}} I^{\vee}
\end{tikzcd}
\end{center}
\end{thm}

\begin{proof}
Rizov's Theorem \ref{Rizov} gives a unique Hodge isometry $\eta \colon T(X)_{\Q} \rightarrow T(X^{\tau})_\Q$ such that the following diagram (of isomorphisms) commutes
\begin{center} {
\begin{tikzcd}[row sep= large, column sep =large]
\widehat{T}(X)_\Q \arrow[r, "\eta \otimes \widehat{\Z}"] & \widehat{T}(X^\tau)_\Q \\
\widehat{T}(X)_\Q \arrow[u, "\frac{s}{\bar{s}}"]\arrow[ur, "\tau^*"]
\end{tikzcd}
}
\end{center}
If we consider  $\widehat{T}(X) \subset \widehat{T}(X)_\Q$ and $\widehat{T}(X^\tau)\subset \widehat{T}(X^\tau)_\Q$, then the Galois action $\tau^*$ restricts to an isomorphism of $\widehat{\Z}$-lattices $$\tau^* \colon \widehat{T}(X) \xrightarrow{\sim} \widehat{T}(X^\tau).$$
This means that the two lattices $T(X^{\tau})$ and $ \eta \big( \frac{s}{\bar{s}}\widehat{T}(X) \big) \cap T(X^{\tau})_{\Q}$ inside $\widehat{T}(X^\tau)_\Q$ are actually the same. Since both $\eta$ and multiplication by $\frac{s}{\bar{s}}$ are isometries and since $\tau$ fixes the reflex field by assumptions, we must have that the type of $(T(X^{\tau}), \tau_{*}B, \tau^{ad} \circ \iota)$ is $$\frac{s}{\bar{s}}  \cdot \bigg( I, \alpha, G, \sigma \bigg).$$
Choose a map $\Phi^{'}_{X^{\tau}}$ realizing the above type for $X^{\tau}$. \\
\textbf{Claim:} there exists a unique root of unity $\mu \in \mathcal{O}^{\times}_{E}$ such that the following commutes 

\begin{center} {
\begin{tikzcd}[row sep= large, column sep =large]
\widehat{T}(X) \arrow[r, "\tau^*"] \arrow[d, "\Phi_{X} \otimes \widehat{\Z}"]
& \widehat{T}(X^\tau)\arrow[d, "\Phi^{'}_{X^{\tau}} \otimes \widehat{\Z}"] \\
I \otimes \widehat{\Z} \arrow[r, "\frac{s}{\bar{s}} \mu"']
& \frac{s}{\bar{s}}I \otimes \widehat{\Z}.
\end{tikzcd} }
\end{center}
Indeed, consider the following 

\begin{center} {
\begin{tikzcd}[row sep= large, column sep =large]
T(X)_{\Q} \arrow[r, "\eta"] \arrow[d, "\Phi_{X} \otimes \Q"]
& T(X^{\tau})_{\Q} \arrow[d, "\Phi^{'}_{X^{\tau}} \otimes \Q"] \\
E = I \otimes \Q \arrow[r, dashrightarrow]
& \frac{s}{\bar{s}}I\otimes \Q = E,
\end{tikzcd} }
\end{center}
We can complete the dashed arrow uniquely with multiplication by some element $\mu \in E^{\times}$ with $\mu \overline{\mu} =1$, since $\eta$ is a Hodge isometry. We can enlarge the diagram above as 
\begin{center} {

\begin{tikzcd}[row sep=scriptsize, column sep=scriptsize]
&  \widehat{T}(X)_\Q \arrow[dl, "\frac{s}{\bar{s}}"'] \arrow[dr, "\tau^*"] \arrow[dd, "\Phi_{X} \otimes \A_f" near end] \\
 \widehat{T}(X)_\Q\arrow[rr, crossing over, "\eta \otimes \A_f" near start] \arrow[dd, "\Phi_{X} \otimes \A_f"'] & &  \widehat{T}(X^\tau)_\Q \\
& I_{\A_f} \arrow[dl, "\frac{s}{\bar{s}}"] \arrow[dr, "\frac{s}{\bar{s}}\mu"]\\
 I_{\A_f} \arrow[rr, "\mu"] & & \frac{s}{\bar{s}} I_{\A_f} \arrow[from=uu, "\Phi^{'}_{X^{\tau}} \otimes \A_f" ].\\
\end{tikzcd} }
\end{center}
One can show by diagram chasing that $\mu \big(\frac{s}{\bar{s}}\widehat{I} \big) = \frac{s}{\bar{s}} \widehat{I},$ so that $\mu \in \widehat{\Oo_E} \cap E = \mathcal{O}_E$, and the condition $\mu \overline{\mu} = 1$ forces $\mu$ to be a root of unity. Put $\Phi_{X^{\tau}}   \coloneqq \mu \cdot \Phi^{'}_{X^{\tau}}$. We obtain another commutative diagram analogous to the one above
\begin{center}
\begin{tikzcd}[row sep=scriptsize, column sep=scriptsize]
& \widehat{T}(X)_\Q\arrow[dl, "\frac{s}{\bar{s}}"'] \arrow[dr, "\tau^*"] \arrow[dd, "\Phi_{X} \otimes \A_f" near end] \\
\widehat{T}(X)_\Q\arrow[rr, crossing over, "\eta \otimes \A_f" near start] \arrow[dd, "\Phi_{X} \otimes \A_f"'] & &  \widehat{T}(X^\tau)_\Q \\
& I_{\A_f} \arrow[dl, "\frac{s}{\bar{s}}"'] \arrow[dr, "\frac{s}{\bar{s}}"]\\
 I_{\A_f} \arrow[rr, "1"] & & \frac{s}{\bar{s}} I_{\A_f} \arrow[from=uu, "\Phi_{X^{\tau}} \otimes \A_f" ],\\
\end{tikzcd}
\end{center} 
so that $\Phi_{X^{\tau}}$ is the required map. The unicity comes from Lemma \ref{Uniq}. 
\end{proof}

\section{K3 class group and K3 class field} \label{section K3 class groups}
Before starting this section, let us fix some notations from algebraic number theory that we are going to use through the rest of this paper. Let $E / F$ be a cyclic extension of number fields and write $G:=\Gal(E/F) = \langle \sigma \rangle $. In this section, $E$ will always be a CM field and $F$ its maximal totally real subfield, but in Section \eqref{section cardinality K3 class groups} it will just be a general cyclic extension and most of these notations will not be used until then. Let $ I \subset \Oo_E$ be an ideal, and consider
\begin{itemize}
\item $\mathcal{I}_E$ the group of fractional ideals of $E$;
\item $\mathcal{I}^I_E \subset \mathcal{I}_E$ the group of fractional ideals coprime to $I$;
\item $E^I   \coloneqq \{ e \in E^{\times} \colon e\Oo_E \in \mathcal{I}^I_E \}$;
\item $E^{I,1}   \coloneqq \{ e \in E^{\times} \colon v(e-1) \geq v(I) \,\, \forall \,\, \text{finite place} \,\, v \,\, \text{such that} \,\, v(I) > 0 \}; $
\item $\Oo^I_E   \coloneqq \Oo^{\times}_E \cap E^{I,1} $;
\item $\mathcal{P}^I_{E}   \coloneqq \{ e\Oo_E \colon e \in E^{I,1} \} \subset \mathcal{I}^I_E;$
\item $\Cl_{I}(E)   \coloneqq \mathcal{I}^I_E  / \mathcal{P}^I_{E}$ the ray class group modulo $I$; 
\item An \textit{invariant} ideal is an ideal such that $\sigma(I) = I$;
\item If $I$ is invariant then $\Cl'_I(E)   \coloneqq \Cl_I(E) / \Cl_I(E)^G $. In particular $\Cl'(E)   \coloneqq \Cl(E) / \Cl(E)^G$;
\item $\mathrm{N} \colon E^\times \to F^\times$ the norm morphism.
\item If $I \subset \Oo_E$ is a proper ideal, its support is $$S(I)   \coloneqq \{ \mathfrak{p} \,\, \text{prime ideal of} \,\, E \colon I \subset p\}.$$
\item If $\mathfrak{m}$ is a modulus for $\Oo_F$, i.e. a formal product of a proper ideal and archimedean valuations, we will denote by $e(E/F, \mathfrak{m})  \coloneqq \prod_{v \nmid \mathfrak{m}}e(v)$, where the product is taken over all the places (both finite and archimedean) of $F$ that do not divide $\mathfrak{m}$ and $e(v)$ denotes their ramification index in the field extension $E/F$;
\item Let $E$ be any number field, for every ideal $I \subset \Oo_E$ we denote by $\phi_E(I)   \coloneqq | (\Oo_E / I)^{\times}|$ the associated Euler's totient function. 
\end{itemize}

Given a CM number field $E$, Theorem \ref{MTSHIMURA} suggests the introduction of a class group (as meant in Chapter 9 of Shimura's book \cite{shimura2014arithmetic}), the K3 class group $G_{K3}(E)$ of $E$, and of its related class field, an abelian extension of $E$ obtained via class field theory, with Galois group isomorphic to $G_{K3}(E)$. These objects will be of essential use later on, especially in the computations of the fields of moduli in the next section. In order to introduce them, we recall that by $$ \U_{E} \subset \Res_{E / \Q}(\mathbb{G}_{m}) $$
we mean $E$-linear unitary group, cut out by the equation $e \bar{e} =1$. 

\begin{defi}
Let $E$ be a CM number field. We define the K3 class group of $E$ to be the double coset
$$G_{K3}(E)   \coloneqq \U_{E}(\Q) \backslash \U_{E}(\A_{f}) / \tilde{\U},$$
where $\tilde{\U}$ is the subgroup generated by all the $u \in  \U_{E}(\A_{f})$ such that for every finite place $v$, $u_{v}$ is a unit, i.e., $$\tilde{\U} = \{ u \in \U_{E}(\A_{f}) \colon u \mathcal{O}_{E} = \mathcal{O}_{E} \} = \{ u \in \U_{E}(\A_{f}) \colon u \in \widehat{\mathcal{O}}_{E}^\times \}.$$
\end{defi}

There is a canonical, continuous map from the finite id\`eles of $E$ to $G_{K3}(E)$, namely
\begin{align} \label{maps idele Gk3}
\A_{E,f}^\times &\to G_{K3}(E) \\
s &\mapsto \frac{s}{\bar{s}} \nonumber,
\end{align}
which is a surjection due to Hilbert's Theorem 90 for id\`eles. 
\begin{defi} \label{SE}
The kernel of the above map $\A_{E,f}^\times \to G_{K3}(E) $ is denoted by $S_{E}$, and it corresponds to $$S_{E} = \{ s\in \A_{E,f}^\times \colon \exists e \in U_{E}(\Q) \colon e \frac{s}{\bar{s}} \mathcal{O}_{E} = \mathcal{O}_{E}  \}$$
\end{defi}
Note that also $E^\times \subset S_{E}$. 
\begin{defi}
The abelian extension of $E$ obtained via class field theory from the subgroup $ S_{E}$ of $\A_{E,f}^\times$ is called the K3 class field of $E$ and it is denoted by $F_{K3}(E)$. 
\end{defi}
Understanding these class fields (the one just introduced and the others to come) will occupy the next two sections. The first step is to relate them to ray class fields, i.e. to abelian extensions of $E$ that we already know. 
\begin{prop} \label{better understanding}
Denote by $K(E)$ the Hilbert class field of $E$ and by $K'(E)$ the subextension of $K(E)$ with Galois group $\cong \Cl'(E)$. There is a diagram of abelian extensions

\begin{center}
\begin{tikzcd}[every arrow/.append style=dash]

& K(E) 
 \arrow{dd} & \\
 F_{K3}(E) 
  \arrow{dr} & \\
& K' (E)
 \arrow{d}\\
& E 

\end{tikzcd}

\end{center}
with $$\Gal( F_{K3}(E) / K'(E)) \cong \frac{\mathcal{O}^{\times}_{F} \cap \mathrm{N}(E^{\times}) }{ \mathrm{N}(\Oo^{\times}_{E})}$$. \\

\end{prop}

\begin{proof}

Indeed, consider the group 
$$\tilde{S}_{E} = \lbrace s \in \A^{\times}_{E,f} \colon \exists e\in E^{\times} \colon e\frac{s}{\bar{s}} \mathcal{O}_{E} =  \mathcal{O}_{E} \rbrace$$
Clearly, $S_{E} \subset \tilde{S}_{E}$. To understand the quotient $\tilde{S}_{E} / S_{E}$, let $s \in \tilde{S}_{E}$ and consider $e \in E^\times$ such that $e \frac{s}{\bar{s}} \mathcal{O}_{E} =  \mathcal{O}_{E} $. We must have $(e \bar{e}) = \mathcal{O}_{E}$, i.e. $e \bar{e} \in \mathcal{O}^{\times}_{F} \cap \mathrm{N}(E^{\times})$. If $e' \in E^{\times}$ is another element such that $e'\frac{s}{\bar{s}} \mathcal{O}_{E} =  \mathcal{O}_{E} $, then $e'$ and $e$ differ by a unit, $e' = e u$ with $u \in \mathcal{O}_{E}^\times$, and $e' \overline{e'} = u \overline{u} e \bar{e}$. We have thus a well-defined map 

\begin{align} \label{map to U plus}
f \colon \tilde{S}_{E} &\to \frac{\mathcal{O}^{\times}_{F} \cap \mathrm{N}(E^{\times}) }{ \mathrm{N}(\Oo^{\times}_{E})} \\
s &\mapsto e \bar{e}. \nonumber
\end{align}

Note that $ \frac{\mathcal{O}^{\times}_{F} \cap \mathrm{N}(E^{\times}) }{ \mathrm{N}(\Oo^{\times}_{E})} $ is a finite $2-$torsion abelian group. Hence it is isomorphic to $(\Z / 2 \Z)^{n}$ for some $n \in \mathbb{N}$. The map $f$ is surjective: let $x \in \mathcal{O}^{\times}_{F} \cap \mathrm{N}(E^{\times}) $ and write $x = y \overline{y}$ with $y \in E^{\times}$. By Hilbert's theorem 90 for ideals (see \cite{classicalinvitationtonumbertheory}, p. 284) we can find a fractional ideal $I$ such that $I / \bar{I} = (y)$ (we take the freedom to write $I / \bar{I} $ for $I \cdot \bar{I}^{-1}$). Then for any $s \in \A^{\times}_{E,f}$ with $s\Oo_E = I$ one has $f(s) = x$. \\
\textbf{Claim:} the kernel of the map \eqref{map to U plus} is $S_{E}$. \\
Indeed, $s \in \tilde{S}_{E} $ is in the kernel if and only if there exists $e \in E^{\times}$ such that $e\frac{s}{\bar{s}} \mathcal{O}_{E} = \mathcal{O}_{E}$ and $e \bar{e} = u \overline{u}$ for some $u \in \mathcal{O}^\times_{E}$. But consider now $e'   \coloneqq \frac{e}{u}$, then clearly also $e'\frac{s}{\bar{s}} \mathcal{O}_{E} = \mathcal{O}_{E}$, and moreover $e' \overline{e'} =1$, i.e. $s \in S_{E}$. The next step, and final one, is to understand to which abelian extension the group $\tilde{S}_{E}$ is associated. Consider the natural projection maps $$ \A_{E,f}^\times \twoheadrightarrow \Cl(E) \twoheadrightarrow \Cl'(E).$$
\textbf{Claim:} the kernel of the above composition is  $\tilde{S}_{E}$. Indeed, $s \in \A^{\times}_{E} $ lies in the kernel if and only if the fractional ideals associated to $s$ and $\bar{s}$ are the same in the class group of $E$, i.e. if and only if exist $e \in E^\times$ such that $e\frac{s}{\bar{s}} \mathcal{O}_{E} = \mathcal{O}_{E}$. This completes the proof.
\end{proof}
In particular
\begin{equation} \label{divisibility relations}
| G_{K3}(E) | =  [ \mathcal{O}^{\times}_{F} \cap \mathrm{N}(E^{\times}) \colon \mathrm{N}(\Oo^{\times}_{E})] \cdot  | \Cl'(E)|.
\end{equation}

\begin{rmk}
If $E$ is imaginary quadratic, then $$\frac{\mathcal{O}^{\times}_{F} \cap \mathrm{N}(E^{\times}) }{ \mathrm{N}(\Oo^{\times}_{E})} = 1$$
\end{rmk}
To generalize the constructions above, one fixes an ideal $I \subseteq \mathcal{O}_{E}$ with prime decomposition $$I = \prod_i^k \mathfrak{p}_i^{n_i}$$ and denote by 

$$  \tilde{\U}_I   \coloneqq \{u \in  \U_{E}(\A_{f}) \colon u \in \widehat{\Oo}_E^\times \,\, \text{and} \,\, u_{\mathfrak{p}_i} \in 1 + \mathfrak{p}_i^{n_i} \Oo_{E,\mathfrak{p}_i}  \,\, \text{for each} \,\, i =1, \cdots, k \} .$$
\begin{defi} \label{K3 class field and group}
The K3 class group modulo $I$ is the double quotient 
$$G_{K3, I}(E)    \coloneqq \U_{E}(\Q) \backslash \U_{E}(\A_{f}) / \tilde{\U}_I,$$
and the K3 class field moduli $I$ is the abelian extension $F_{K3,I}(E)$ of $E$ associated to the surjection 
\begin{align*}
\A_{E,f}^\times &\twoheadrightarrow G_{K3,I}(E) \\
s &\mapsto \frac{s}{\bar{s}}.
\end{align*}
\end{defi}
Note that if we put $J   \coloneqq I \cap \overline{I}$, then $$ G_{K3,I}(E) =  G_{K3,J}(E) =  G_{K3,\overline{I}}(E)$$ directly from the definition. So that, without loss of generality, we can assume that $I$ is invariant. The following is the analogue of the previous proposition. 
\begin{prop} \label{better understanding 2}
Denote by $K_I(E)$ the ray class field of $E$ modulo $I$ and by $K_I'(E)$ the subextension of $K_I(E)$ with Galois group $\cong \Cl'_I(E)$. There is a diagram of abelian extensions 

\begin{center}
\begin{tikzcd}[every arrow/.append style=dash]

& K_I(E) 
 \arrow{dd} & \\
 F_{K3,I}(E) 
  \arrow{dr} & \\
& K_I'(E)
 \arrow{d}\\
& E 

\end{tikzcd}

\end{center}
with $$\Gal( F_{K3,I}(E) / K_I'(E) \cong \frac{\Oo^{\times}_F \cap \mathrm{N}(E^{I,1})}{\mathrm{N}(\Oo^{I}_E)}.$$
\end{prop}

\begin{proof}
As before, we begin by studying the kernel of the map $\A_{E,f}^\times \twoheadrightarrow G_{K3,I}(E).$ We denote it by $S_{I}$, so that $$S_{I} = \big\{ s\in \A^{\times}_{E}: \exists u \in \U_{E}(\Q) \colon \frac{s}{\bar{s}} u \mathcal{O}_{E} = \mathcal{O}_{E}, \,  u\frac{s}{\bar{s}} \equiv 1 \mod I \big\}.$$
Denote by $\tilde{S}_{I}$ the group $$\tilde{S}_{I} =\big\{ s\in \A^{\times}_{E} : \exists e \in E^{\times} : \frac{s}{\bar{s}} e \mathcal{O}_{E} = \mathcal{O}_{E}, \,  e\frac{s}{\bar{s}} \equiv 1 \mod I \big\}.$$
We again have an injection 
\begin{equation} \label{equationF}
\tilde{S}_{I} / S_{I} \hookrightarrow \frac{\Oo^{\times}_F \cap \mathrm{N}(E^{I,1})}{\mathrm{N}(\Oo^{I}_E)},
\end{equation}
and we need to prove surjectivity. As in the proof of Proposition \ref{better understanding}, let $x \in \Oo^{\times}_F \cap \mathrm{N}(E^{I,1})$ and let $y \in E^{I,1}$ be such that $y \overline{y} = x$ and find a fractional ideal $J$ of $E$ such that $J /  \bar{J} = (y)$. We need $J$ to be in $\mathcal{I}^I_E$ in order to conclude, so suppose it is not. \\ 
\textbf{Claim:} there exists an invariant fractional ideal $\mathfrak{a}$ such that $ \mathfrak{a} | J$ and $J / \mathfrak{a}$ is coprime to $I$. Indeed, let $\mathfrak{p}$ be a prime ideal of $E$, suppose that $v_\mathfrak{p}(\gcd(I,J)) \neq 0$ and let $n$ be the power of $\mathfrak{p}$ appearing in the factorization of $J$. If $\overline{\mathfrak{p}} = \mathfrak{p}$, then the ideal $J' = J / \mathfrak{p}^n$ has still the property that we need, i.e. $J' / \overline{J'} = (y)$, and $J'$ has no $\mathfrak{p}-$factor in common with $I$. If $\mathfrak{p} \neq \overline{\mathfrak{p}}$, write again $J' = J / \mathfrak{p}^n$ and consider $$(y) = J / \bar{J} =(J' /  \overline{J'}) (\mathfrak{p}^n / \overline{\mathfrak{p}}^n).$$
Since by construction $(y)$ is coprime to $I$ and $I$ is invariant, we must have that $\overline{\mathfrak{p}}$ divides $J'$ exactly with the same exponent $n$, hence $J'' = J / (\mathfrak{p} \overline{\mathfrak{p}})^n$ is still such that $(y) = J'' / \overline{J''}$ and has neither $\mathfrak{p}$ nor $\overline{\mathfrak{p}}$ factors in common with $I$. Doing this for every prime such that $v_\mathfrak{p}(\gcd(I,J)) \neq 0$, we find an ideal $J$ coprime to $I$ with $J / \bar{J} = (y)$. Therefore, the claim follows and \eqref{equationF} is surjective. Exactly as before, we recover $\tilde{S}_I$ as the kernel of the natural projection  
$$ \A^{\times}_{E,f} \twoheadrightarrow  \Cl'_I(E),$$
and this concludes the proof. 
\end{proof}
In particular 
\begin{equation} \label{divisibility relations 2}
 |G_{K3,I}(E)|  =    [\Oo^{\times}_F \cap \mathrm{N}(E^{I,1}) : \mathrm{N}(\Oo^{I}_E)] \cdot  | \Cl'_I(E)|.
\end{equation}
\begin{rmk} \label{quadratic imaginary remark}
When $E$ is imaginary quadratic there are equalities $ F_{K3,I}(E) =  K_I'(E)$ and $  G_{K3,I}(E) = \Cl'_I(E)$.
\end{rmk}

\section{Invariant ideals and K3 class group} \label{section cardinality K3 class groups}

In this section we continue to study the groups $G_{K3,I}(E)$, in particular we compute their cardinality. Most of the results in this section works for every cyclic extension $E/F$, so we rather work in this generality, the proofs being the same. By Theorem \ref{better understanding 2}, we know that $$ | G_{K3,I}(E) | = \frac{|\Cl_I(E)|}{|\Cl_I(E)^G|} \cdot [ \Oo^{\times}_F \cap \mathrm{N}(E^{I,1}) \colon \mathrm{N}(\Oo^{I}_E) ]. $$
When $I=\Oo_E$ we have the following (see Lemma 4.1 of \cite{langcyclo})

\begin{lemma}
Let $E / F$ be a cyclic extension with Galois group $G$. Then 
$$ | \Cl(E)^G | = \frac{h_F \cdot e(E/F)}{[E : F] \cdot [ \Oo^{\times}_F \colon \mathrm{N}(E^\times) \cap \Oo^{\times}_F ]}, $$
where $h_F$ is the class number of $F$ and $$e(E/F)   \coloneqq \prod_v e(v),$$
the product of all the ramification indices over all the places of $F$, both finite and infinite. 
\end{lemma}
Putting this together with Theorem \ref{better understanding 2} leads to 
$$  |G_{K3}(E)| = 2 \cdot \frac{h_E \cdot [\Oo^{\times}_F : \mathrm{N}(\Oo^{\times}_E)]}{h_F \cdot e(E/F)}.$$ 
Using basically the same proof of \cite{langcyclo}, we compute now the cardinalities $| \Cl_I(E)^G |$, where $I$ is any invariant ideal. 
We are going to use the notation introduced at the beginning of the last section. Moreover, for a $G-$ module $M$ we will denote by $$\mathrm{H}^i(M)   \coloneqq \hat{H}^i(G,M),$$
the $i-$th Tate cohomology group and by $Q(M)$ its Herbrand quotient
(when defined). We remind the reader that since $G$ is cyclic the Tate cohomology groups satisfy $\mathrm{H}^{\bullet} \cong \mathrm{H}^{\bullet + 2}.$ 
\begin{lemma} \label{Q}
Let $I$ be an invariant ideal, then 
$$ Q(\Oo_E^I) = Q(\Oo^{\times}_E) = \frac{1}{[E \colon F]} e_{\infty}(E/F), $$
with $$e_{\infty}(E/F) = \prod_{v | \infty} e(v),$$ where the product ranges over all the archimedean valuations of $F$. 
\end{lemma}

\begin{proof}
The equality $Q(\Oo_E^I) = Q(\Oo^{\times}_E)$ descends from the fact that $\Oo^{\times}_E / \Oo_E^I$ is a finite group. The second equality of the statement follows from Corollary 2, Theorem 1, Chapter IX of \cite{langANT}.  

\end{proof}

\begin{thm} \label{invariantformula}
Let $I \subset \Oo_E$ be an invariant ideal and denote by $J   \coloneqq I \cap \Oo_{F}$. Then
$$| \Cl_{I}(E)^G | =  \frac{h_J(F) \cdot e(E/F,J)  \cdot | \mathrm{H}^1(E^{I,1}) | }{ [E \colon F] [ \Oo^J_F \colon \mathrm{N}(E^{I,1}) \cap \Oo_{F}^{\times}  ] } $$
where $$ e(E/F, J) = \prod_{v \nmid J} e(v).$$ 
\end{thm}

\begin{proof}
Consider the short exact sequence defining the ray class group 
$$0 \rightarrow \mathcal{P}^I_{E} \rightarrow \mathcal{I}^I_E \rightarrow \Cl_I(E) \rightarrow 0.$$
Taking $G-$invariants we obtain 
$$0 \rightarrow \mathcal{P}^{I,G}_{E} \rightarrow \mathcal{I}^{I,G}_E \rightarrow \Cl_I(E)^G \rightarrow \mathrm{H}^1(\mathcal{P}^{I}_{E}) \rightarrow 0,$$
since $\mathrm{H}^1(\mathcal{I}^I_E) = 0$
and therefore
\begin{equation} \label{eq1}
| \Cl_{I}(E)^G | = [\mathcal{I}^{I,G}_E \colon \mathcal{P}^{I,G}_{E}] \cdot |\mathrm{H}^1(\mathcal{P}^{I}_{E})|.
\end{equation}
Now, we compute the two numbers on the right-hand side. For the first one, write 
\begin{equation}\label{eq2}
 [\mathcal{I}^{I,G}_E \colon \mathcal{P}^{I,G}_{E}] = \frac{ [\mathcal{I}^{I,G}_E \colon \mathcal{P}^{J}_{F}]}{[\mathcal{P}^{I,G}_{E} \colon \mathcal{P}^{J}_{F}]} = \frac{ [\mathcal{I}^{I,G}_E \colon \mathcal{I}^J_F] \cdot [ \mathcal{I}^J_F \colon \mathcal{P}^{J}_{F}] }{[\mathcal{P}^{I,G}_{E} \colon \mathcal{P}^{J}_{F}]} = \frac{e(E/F, \infty\cdot J) \cdot h_{J}(F)}{{[\mathcal{P}^{I,G}_{E} \colon \mathcal{P}^{J}_{F}]}}.
\end{equation}
In order to find $[\mathcal{P}^{I,G}_{E} \colon \mathcal{P}^{J}_{F}],$ one takes the $G-$invariants of the next exact sequence 
$$0 \rightarrow \Oo_E^I \rightarrow E^{I,1} \rightarrow \mathcal{P}^{I}_{E} \rightarrow 0,$$
to obtain  
$$0 \rightarrow \Oo_F^J \rightarrow F^{J,1} \rightarrow \mathcal{P}^{I,G}_{E} \rightarrow \mathrm{H}^1(\Oo_E^I) \rightarrow \mathrm{H}^1( E^{I,1} ).$$
Thus, if $H \subset \mathrm{H}^1( E^{I,1} )$ denotes the image of the last map, one has
$$[\mathcal{P}^{I,G}_{E} \colon \mathcal{P}^{J}_{F}] = \frac{|\mathrm{H}^1(\Oo_E^I)|}{|H|} = \frac{|\mathrm{H}^0(\Oo_E^I)|}{|H| \cdot Q(\Oo_E^I)}. $$
By Lemma \ref{Q} we know the value of $Q(\Oo_E^I)$, and by definition 
$$|\mathrm{H}^0(\Oo_E^I)| = [\Oo_F^J \colon \mathrm{N}(\Oo_E^I) ],$$
so we conclude that
\begin{equation} \label{eq3} 
[\mathcal{P}^{I,G}_{E} \colon \mathcal{P}^{J}_{F}] =  \frac{|\mathrm{H}^1(\Oo_E^I)|}{|H|} = \frac{[\Oo_F^J \colon \mathrm{N}(\Oo_E^I) ]}{|H| \cdot Q(\Oo_E^I)}. 
\end{equation}
Concerning the second number $ |\mathrm{H}^1(\mathcal{P}^{I}_{E})|$ we use the exact sequence 
$$ 0 \rightarrow \mathrm{H}^1(E^{I,1}) / H \rightarrow \mathrm{H}^1( \mathcal{P}^{I}_{E} ) \rightarrow \mathrm{H}^0( \Oo_E^I ) \rightarrow \mathrm{H}^0( E^{I,1} ),$$
to see that 
\begin{equation} \label{eq4}
| \mathrm{H}^1( \mathcal{P}^{I}_{E} )| = \frac{ |\mathrm{H}^1(E^{I,1})|}{| H |} \cdot | \ker( \mathrm{H}^0( \Oo_E^I ) \rightarrow \mathrm{H}^0( E^{I,1} )  )|.
\end{equation}
(Note how the term $|H|$ above and the one in \eqref{eq3} will cancel each other in the final formula). Now, 
$$\ker( \mathrm{H}^0( \Oo_E^I ) \rightarrow \mathrm{H}^0( E^{I,1} )) \cong (\mathrm{N}(E^{I,1}) \cap \Oo^{\times}_F) / \mathrm{N}(\Oo^{I}_E).$$
Using the inclusions $$\mathrm{N}(\Oo^{I}_E) \subset \mathrm{N}(E^{I,1}) \cap \Oo^{\times}_F \subset \Oo^J_F$$ and putting equations \eqref{eq1}, \eqref{eq2}, \eqref{eq3} and \eqref{eq4} together, we conclude. 
\end{proof}
Together with Theorem \ref{better understanding 2} this implies the following. 
\begin{cor} \label{superformula}
Let $E$ be a CM number field, $F$ its maximal, totally real subextension, and $I \subset \Oo_E$ an invariant ideal. Then 
$$ |G_{K3,I}(E)| = \frac{ 2 \cdot h_E \cdot \phi_E(I) \cdot [\Oo^{\times}_F \colon \mathrm{N}(\Oo^I_E)]}{h_F \cdot \phi_F(J) \cdot [\Oo^{\times}_E \colon \Oo^I_E] \cdot e(E/F, J) \cdot | \mathrm{H}^1(E^{I,1}) |}.$$
\end{cor}
\begin{proof}
This follows from Theorem \ref{invariantformula} and Theorem \ref{better understanding 2}, using the well-known fact $$h_{I}(E) = h_E \frac{\phi_E(I)}{[ \Oo^{\times}_E \colon \Oo^I_E ]}. $$
\end{proof}
The only mysterious term appearing in Theorem \ref{superformula} is $| \mathrm{H}^1(E^{I,1}) |.$ Note that this group is always $2-$torsion and finitely generated. We have the following partial result:
\begin{prop} \label{H1}
In the assumptions of Theorem \ref{superformula}
\begin{enumerate}
\item If $\gcd(2,I) = (1)$. Then $\mathrm{H}^1(E^{I,1}) = 0;$
\item Write $I= I_2 \cdot I'$ with $I'+ (2)=(1)$, and likewise put $J = J_2 \cdot J'$. There is a natural left exact sequence $$1 \rightarrow \frac{ (\Oo_E / I_2)^{\times, G}}{ (\Oo_F / J_2)^{\times}} \rightarrow \mathrm{H}^1(E^{I,1}) \rightarrow \bigoplus_{q \in S(J_2)} \Z / e(q) \Z. $$ 
\item If every prime ideal dividing $J_2$ does not ramify in $E$, then  $\mathrm{H}^1(E^{I,1}) = 0.$
\end{enumerate}

\end{prop}
\begin{proof} \mbox{}
\begin{enumerate}
\item Let $x \in E^{I,1}$ be such that $x \overline{x} =1$. Then, if we put $y = 1/2 + x/2$, we also have $y \in E^{I,1}$ (since by assumptions $2$ and $I$ are coprime) and $y / \overline{y} =x$. 

\item We start by understanding the quotient $Q_I$ of 
\begin{equation} \label{sesQ}
1 \rightarrow E^{I,1} \rightarrow E^{I',1} \rightarrow Q_I \rightarrow 1.
\end{equation} 
In order to do this, consider the morphism of short exact sequences 
 \begin{center}
\begin{tikzcd} 
 1 \arrow[r] 
& E^{I,1} \arrow[r] \arrow[d] 
& E^{I} \arrow[r] \arrow[d] 
& (\Oo_E / I)^{\times} \arrow[r] \arrow[d]
& 1 \\
1 \arrow[r] 
& E^{I',1} \arrow[r] 
& E^{I'} \arrow[r] 
& (\Oo_E / I')^{\times} \arrow[r]
& 1,
\end{tikzcd}
\end{center}
The following sequence 
$$ 1 \rightarrow E^I \rightarrow E^{I'} \xrightarrow{\oplus v_\mathfrak{p}} \bigoplus_{\mathfrak{p} \in S(I_2)} \Z \rightarrow 0  $$
is exact, due to the theorem on the independence of valuations. 
Then via the snake lemma we obtain 
\begin{equation} \label{sesforQI}
 1 \rightarrow (\Oo_E / I_2)^{\times} \rightarrow Q_I \rightarrow \bigoplus_{\mathfrak{p} \in S(I_2)} \Z \rightarrow 0.
\end{equation}
We can do the same over $F$, obtaining analogous results: we have two exact sequences 
$$1 \rightarrow F^{J,1} \rightarrow F^{J',1} \rightarrow Q_J \rightarrow 1$$
and 
\begin{equation} \label{sesforQJ}
1 \rightarrow (\Oo_F / J_2)^{\times} \rightarrow Q_J \rightarrow \bigoplus_{q \in S(J_2)} \Z \rightarrow 0.
\end{equation}
Taking Galois invariants of \eqref{sesQ} and using the first point of this proposition, we obtain 
$$
1 \rightarrow F^{J,1} \rightarrow F^{J',1} \rightarrow Q_I^G \rightarrow \mathrm{H}^1(E^{I,1}) \rightarrow 1.
$$ Thus, we can identify 
\begin{equation} \label{identification}
\mathrm{H}^1(E^{I,1})  \cong \coker(Q_J \rightarrow Q_I^G ). 
\end{equation}
Applying the snake lemma again to the following diagram

 \begin{center}
\begin{tikzcd} 
 1 \arrow[r] 
& (\Oo_F / J_2)^{\times} \arrow[r] \arrow[d] 
& Q_J \arrow[r] \arrow[d] 
& \bigoplus_{q \in S(J_2)} \Z \arrow[r] \arrow[d]
& 0 \\
1 \arrow[r] 
& (\Oo_E / I_2)^{\times,G} \arrow[r] 
& Q^G_I \arrow[r] 
&\big( \bigoplus_{\mathfrak{p} \in S(I_2)} \Z \big)^G,
\end{tikzcd}
\end{center}
we obtain
\begin{equation} 
1 \rightarrow \frac{ (\Oo_E / I_2)^{\times, G}}{ (\Oo_F / J_2)^{\times}}  \rightarrow \mathrm{H}^1(E^{I,1}) \rightarrow \bigoplus_{q \in S(J_2)} \Z / e(q) \Z. 
\end{equation}
This concludes the proof of point 2.
\item Under these assumptions $$ \frac{ (\Oo_E / I_2)^{\times, G}}{ (\Oo_F / J_2)^{\times}}  \cong \mathrm{H}^1(E^{I,1}).$$ 
However, since the primes in $ S(J_2)$ do not ramify, $(\Oo_E / I_2)^{\times, G} =  (\Oo_F / J_2)^{\times}$.

\end{enumerate}

\end{proof}

\section{Fields of moduli and applications} 
In this section we compute the field of moduli of the tuple $(T(X), B, \iota)$. This should be interpreted as the field of moduli of the transcendental motive of $X$, together with the cycles in $E(X)$ and some additional Brauer classes. 
\begin{defi}
The field of moduli of $(T(X), B, \iota)$ is the fixed field of 
$$ \{ \sigma \in \Aut(\C / \Q) \colon \text{exists an isomorphism} \,\, (T(X), B, \iota) \cong (T(X^\sigma), \sigma_*(B), \sigma^{ad} \circ \iota) \},$$
where an isomorphism $ (T(X), B, \iota) \cong (T(X^\sigma), \sigma_*(B), \sigma^{ad} \circ \iota) $ is as in Definition \ref{iso k3 tuples}.
\end{defi}
\begin{rmk}\label{remark field of moduli}
Note that if we denote by $M$ the field of moduli of $(T(X), B, \iota)$, then we must have $E \subset M$ because of Lemma \ref{lemma:nuovo}, so that we can `work over $E$'. 
\end{rmk}
\begin{thm}[Field of moduli] \label{field of moduli}
Let $(X, B, \iota)$ be a principal CM $K3$ surface over $\C$ with level structure $B \subset \Br(X)$ and let $E \subset \C$ be its reflex field. Suppose that $(T(X), B, \iota)$ is of type $(I, \alpha, J, \sigma)$ and put $I_B   \coloneqq I^\vee J^{-1} \subset \Oo_E$. Then the field of moduli of $(T(X), B, \iota)$ corresponds to the K3 class field $F_{K3,I_B}(E)$ modulo the ideal $I_B$. Moreover, if $[E \colon \Q] \leq 10$, the field $F_{K3}(E)$ is equal to the field of moduli of $(X, \iota)$. 
\end{thm}

\begin{proof}
Thanks to the remark we need to compute the fixed field of
$$ \{ \sigma \in \Aut(\C / E) \colon \exists \,\, \text{Hodge isometry} \,\, f \colon T(X) \rightarrow T(X^\sigma)\colon f_{*} \circ \sigma^*|_{B} = \id \}.$$
Thanks to Proposition \ref{type thm} and Theorem \ref{MTSHIMURA}, an element  $\tau \in \Aut(\C / E) $ is in the above group if and only if we can find $s \in \A_{E,f}^{\times}$ and $e \in E^{\times}$ such that 
\begin{enumerate}
\item $\art(s) = \tau|_{E^{ab}}$;
\item $\frac{s}{\bar{s}} (I , \alpha , \sigma) \cong (I , \alpha , \sigma)$, i.e. $e \frac{s}{\bar{s}} I = I$ and $e \bar{e} = 1$;
\item The composition $E/I^\vee \xrightarrow{s / \bar{s}} E/\frac{s}{\bar{s}} I^\vee \xrightarrow{e} E/I^\vee$ restricts to the identity on $J/ I^\vee$. 
\end{enumerate}
Via class field theory, this corresponds to 
$$ \lbrace s \in \A^{\times}_{E,f} \colon \exists e\in E^{\times} \colon e \bar{e} = 1 \, , \, e\frac{s}{\bar{s}} \Oo_E = \Oo_E \, ,  e\frac{s}{\bar{s}} \equiv 1 \mod I_B \,  \rbrace, $$
and we recognize this group to be exactly the kernel of $\A^{\times}_{E,f} \twoheadrightarrow G_{K3,I_B}(E)$ (see Proposition \ref{better understanding 2}). The last assertion follows from the second remark after Proposition \ref{key proposition}. 
\end{proof}
The immediate corollary we get is 
\begin{cor} \label{baruerCM}
Let $X/K$ be a principal K3 surface with CM over $K$ and let $I \subset \Oo_E$ be the unique ideal such that $\Br(\overline{X})[I] = \Br(\overline{X})^{G_K}$. Then $$F_{K3,I}(E) \subset K.$$
\end{cor}

\subsection{Applications to Brauer groups} \label{section applications brauer}

One of the consequences of the finiteness result in \cite{MR3830546} is that \textit{a posteriori} for a fixed number field $K$, there are only finitely many groups that can appear as $\Br(\overline{X})^{G_K}$, where $X/K$ is any K3 surface with CM over $\overline{K}$. In this last section we show how our previous results can be applied to produce a computable bound for the Galois fixed part of Brauer groups of principal CM K3 surfaces. As mentioned in the introduction, there is an algorithm that from a number field $K$ and a CM field $E$ returns a finite set of groups $\Br(E,K)$ such that for every principal CM K3 surfaces $X/K$ with reflex field $E$ one has $$\Br(\overline{X})^{G_K} \in \Br(E,K).$$ It works as follows:
\begin{enumerate}
\item Replace $K$ by $KE$;
\item Find all the invariant ideals $I \subset \Oo_E $ such that 
$$| G_{K3,I}(E)| \divides [K \colon E].$$
This is possible thanks to Theorem \ref{superformula} and Proposition \ref{H1}, which also says that there are finitely many such ideals. Denote them $I_1, \cdots I_n.$ 
\item Now by Theorem \ref{baruerCM} one knows that
$$\Br(\overline{X_K})^{G_K} \cong \Oo_E / I_B,$$
with $I_B \subset \Oo_E$ an ideal amongst the $I_i$'s, hence, we have an inclusion (of isomorphism classes of $\Oo_E$-modules)
$$ \{ \Br(\overline{X})^{G_K} \colon X / K \,\, \text{has CM by} \,\,  \Oo_E \} \subset \{ \Oo_E / I_i \colon  \,\, \text{for} \,\, i=1, \cdots, n \},$$
and we define the latter set to be $\Br(K,E)$. 
\end{enumerate}
\begin{rmks}
\begin{itemize}
\item This strategy is the same one employed by Silverberg in \cite{MR3778184} for torsion points of CM abelian varieties. 
\item In particular, if we put $C  \coloneqq \max_i | \Oo_E/ I_i |$ we have the bound 
$$ | \Br(\overline{X})^{G_K} | \leq C$$
for every principal K3 surface $X / K$ with CM by $E$ over $\overbar{K}.$ 
\end{itemize}
\end{rmks}
In the following, we provide two examples of the algorithm above, both concerning K3 surfaces of maximal Picard rank. 
\begin{enumerate}
\item (Gaussian integers) Let $E=\Q(i)$. In this case, the K3 class field of $E$ is $E$ itself. We put $K=E$. Every invariant ideal of $E$ can be written as $ I = (1+i)^k \cdot (n)$ with $n \in \Z$ and $(n,2) =1$, and we have to find all such $I$ with $G_{K3,I}(E) = 1$. To do so, decompose $$ n = p_1^{\alpha_1} \cdots p_l^{\alpha_l} \cdot q_1^{\beta_1} \cdots q_j^{\beta_j},$$
where the $q$'s are inert (i.e. $ \equiv 3 \mod 4$) and the $p$'s are split (i.e. $\equiv 1 \mod 4$). By Theorem \ref{superformula} we have
$$| G_{K3,I}(E) | = \frac{h_E \cdot \phi_E(I) \cdot [\Oo^{\times}_F \colon \mathrm{N}(\Oo^I_E)]\cdot [E \colon F]}{h_F \cdot \phi_F(J) \cdot [\Oo^{\times}_E \colon \Oo^I_E] \cdot e(E/F, J) \cdot | \mathrm{H}^1(E^{I,1}) |} = $$ $$=\frac{ \phi_E(I) \cdot 4}{\phi_F(J) \cdot [\Oo^{\times}_E \colon \Oo^I_E] \cdot e(E/F, J) \cdot | \mathrm{H}^1(E^{I,1}) |}  $$
\begin{itemize}
\item If $k=0$ and $n > 1$ 
\begin{itemize}
\item $[\Oo^{\times}_E \colon \Oo^I_E] = 4$;
\item $ e(E/F, J) = 4 $, since only $2$ and the place at infinity ramify;
\item $| \mathrm{H}^1(E^{I,1}) | = 1$, by Proposition \ref{H1}.
\end{itemize}
So we obtain 
$$| G_{K3,I}(E) |  = \frac{\phi_K(n) }{4 \cdot \phi(n)} = \frac{1}{4} \prod p_i^{\alpha_i -1}(p_i -1) \cdot \prod q_i^{\beta_i -1}(q_i +1),$$
hence, in this case, $| G_{K3,I}(E) | =1$ if and only if $n= 3$ or $n = 5$. 
\item 
If $k > 0$ and $n = 1$ then
\begin{itemize}
\item $e(E/F, J) = 2$;
\item $ [\Oo^{\times}_K \colon \Oo^{(1+i)^k}_K] = 
\begin{cases}
1 \,\,\text{if} \,\, k=1 \\
 2 \,\,\text{if} \,\, k=2 \\   
 4 \,\,\text{if} \,\, k>2;
\end{cases}$
\item $ \frac{\phi_K(1+i)^k}{\phi((1+i)^k \cap \Z)} = 2^{\lfloor \frac{k}{2} \rfloor}.$ 
\end{itemize}
In particular, we can write $[\Oo^{\times}_K \colon \Oo^{(1+i)^k}_K] = 2^a$ with $a = 0,1,2$.
Since $2$ ramifies in $E$, in general the cohomology groups $\mathrm{H}^1(E^{(1+i)^n,1})$ are not zero. However, Proposition \ref{H1} tells us that their cardinality $| \mathrm{H}^1(E^{(1+i)^n,1}) |$ always divides $$ 2 \cdot  [ (\Oo_K / (1+i)^k )^{\times, G} \colon (\Z / (1+i)^k \cap \Z )^{\times}],$$
and we compute
\begin{itemize}
\item $ [ (\Oo_K / (1+i)^k )^{\times, G} \colon (\Z / (1+i)^k \cap \Z )^{\times}] = 
\begin{cases}
2 \,\,\text{if} \,\, k \,\, \text{is even,} \\
 1 \,\,\text{if} \,\, k \,\, \text{is odd.}
\end{cases} $
\end{itemize}
Write $ | \mathrm{H}^1(E^{(1+i)^n,1}) | = 2^b$ with $b=0,1,2$. Thus, $|G_{K3, (1+i)^k}(E)| = 1$ if and only if $$ \frac{ 2^{\lfloor \frac{k}{2} \rfloor +1 }}{ 2^{a + b }}=1.$$
Since $a+b \leq 4$ this can happen only if $k \leq 7$. Nevertheless, if $k=7$ then $b \leq 1$, so that $k \leq 6$.
\item Finally, assume that $k \geq 1$ and $n>2$. Thanks to the results above, if $|G_{K3, I}| = 1$, then $I = (1+i)^k \cdot 5^{\alpha}$ or $I= (1+i)^k\cdot 3^{\beta}$. Let us begin with the former case: we have $$|G_{K3, I}(E)| = \frac{2^{\lfloor \frac{k}{2} \rfloor} \cdot 3^{\beta-1} \cdot 4 \cdot 4 } {4 \cdot 2 \cdot |\mathrm{H}^1|} = \frac{2^{1+ \lfloor \frac{k}{2} \rfloor} \cdot 3^{\beta-1} } { |\mathrm{H}^1|}.$$
Hence $\beta = 1$ since $\mathrm{H}^1$ is $2-$torsion and as before we see that if $|G_{K3, I}(E)|  = 1$ then necessarily $k \leq 2$. Finally, the same is true for the case $I= (1+i)^k\cdot 5^{\beta}$.

\end{itemize}
Hence, we have the following possibilities for $\Br(\overline{X})^{G_K}$ (as isomorphism classes of abelian groups) 
 $$0,  \, \, \Z /2 , \, \,  (\Z /2)^2 , \,\, \Z / 4  \times \Z /2 , \,\, (\Z /4)^2 , \,\, \Z / 8  \times \Z /4 , $$
 $$(\Z /8)^2 , \,\, (\Z / 3)^2 ,  \, \,  (\Z / 3)^2 \times \Z /2 , \,\, (\Z / 3)^2 \times (\Z /2)^2 ,$$ $$(\Z / 5)^2, \,\, (\Z / 5)^2  \times \Z /2 , \,\, (\Z / 5)^2  \times (\Z /2)^2,$$
confirming the results in \cite{MR3590544} and \cite{MR2802504} about diagonal quartic surfaces.
\item (Eisenstein integers). If $E= \Q( \sqrt{-3})$, the K3 class field of $E$ is again $E$ itself. Put $K=E$. The only prime of $\Z$ that ramifies in $E$ is $3$, with $(3) = (\sqrt{-3})^2 $. In particular, since $2$ does not ramify, thanks to Proposition \ref{H1} we have 
$$| G_{K3,I}(E) | = \frac{h_E \cdot \phi_E(I) \cdot [\Oo^{\times}_F \colon \mathrm{N}(\Oo^I_E)]\cdot [E \colon F]}{h_F \cdot \phi_F(J) \cdot [\Oo^{\times}_E \colon \Oo^I_E] \cdot e(E/F, J) } = \frac{ 4 \cdot \phi_E(I)}{ \phi_F(J) \cdot [\Oo^{\times}_E \colon \Oo^I_E] \cdot e(E/F, J) }$$
for every invariant ideal $I \subset \Oo_E$. As before, let us proceed in computing these numbers. One can check that
\begin{equation} \label{cases ein}  [\Oo^{\times}_E \colon \Oo^I_E] = \begin{cases} 
1 \,\, \text{if} \,\, I = \Oo_E; \\
2\,\, \text{if} \,\, I = (\sqrt{-3}); \\
3 \,\, \text{if} \,\, I = (2); \\ 
6\,\, \text{otherwise}.
\end{cases}
\end{equation}
Write $$I = (\sqrt{-3})^k \cdot p_1^{\alpha_1} \cdots p_l^{\alpha_l} \cdot q_1^{\beta_1} \cdots q_j^{\beta_j},$$
where the $q$'s are inert primes (i.e. $\equiv 2 \mod 3$) and the $p$'s are split (i.e. $\equiv 1 \mod 3$). Hence, 
$$| G_{K3,I}(E) | = 4 \cdot 3^{\lfloor k/2 \rfloor} \cdot \prod p_i^{\alpha_i -1}(p_i -1) \cdot \prod q_i^{\beta_i -1}(q_i +1) \cdot \frac{1}{[\Oo^{\times}_E \colon \Oo^I_E] \cdot e(E/F, J)}.$$ Using this, we see that 
\begin{itemize}
\item If $k=0$, then $| G_{K3,I}(E) | = 1$ if and only if $I= (2), (4), (5), (7)$;
\item If $k=1$, then $| G_{K3,I}(E) | = 1$ if and only if $I = (\sqrt{-3}), (2 \sqrt{-3});$
\item if $k=2$, then $| G_{K3,I}(E) | = 1$ if and only if $I = (3);$
\item if $k=3$, then $| G_{K3,I}(E) | = 1$ if and only if $I = (3 \sqrt{-3});$
\item if $k>3$, then  $| G_{K3,I}(E) | > 1$. 
\end{itemize}
Hence, we have the following possibilities for $\Br(\overline{X})^{G_K}$ (as isomorphism classes of abelian groups) 
$$ 0, \,\, \Z / 3, \,\, (\Z /2)^{ 2} , \,\,  (\Z /2)^2   \times \Z /3, \,\, (\Z /3)^2, \,\, (\Z /4)^2, \,\, \Z / 9 \times \Z /3 , \,\,  (\Z / 5)^2, \,\, (\Z / 7)^2.$$
\end{enumerate}
Note how the latter example was easier due to the vanishing of $H^1(E^{I,1})$ for each $I$. In principle, one could carry on and do similarly for any $E$, but one has to face the difficulty of choosing $K$ accordingly, in order to not obtain empty results. One way to do this when $E$ is quadratic imaginary, and to generalize the computations above, is to choose $K$ to be the Hilbert class field of $E$. In this case in fact there are always K3 surfaces with CM by $\Oo_E$ defined over $K$, as one can simply take $X = \mathrm{Km}(E_1 \times E_2)$, the Kummer surface associated to the product of two elliptic curves $E_1,E_2/ K$ with CM by $\Oo_E$. When $E_1 = E_2$ this was the object of study of Newton's work \cite{MR3483120}. She gives a recipe to explicitly compute the groups $\Br(\overline{X})[\ell^{\infty}]^{G_K}$ for every prime number $\ell$ and any $K$ (see Theorem 2.6 of \textit{loc. cit.}) up to the knowledge of some arithmetic (= class field theoretical) invariant depending on $E$, $K$ and $\ell$. In this last part of the paper, we show how to employ our results to study the aforementioned case when $K$ is the Hilbert class field of $E$. Note that we do not need to make any assumption on the geometry of $X$ (i.e. $X$ does not need to be a Kummer surface, but can be any singular K3 surface with CM by $\Oo_E$), whereas on the other hand we do not compute the Brauer group explicitly, but only list the finitely many possibilities. In the next, c.g. stands for complex conjugation. 

\begin{thm} \label{rachelk}
Let $E$ be a quadratic imaginary field, and let $K = K(E)$ be its Hilbert class field. Then $$\Br(E, K) = \bigg\{ \Oo_E/ I \colon I \,\, \text{is invariant and c.g. acts trivially on} \,\, (\Oo_E / J)^{\times} / \mu(E) \bigg\}.$$
\end{thm}

\begin{rmk}
We could still use the algorithm to get similar results, but as shown in the next proof, the general facts of Section \ref{section K3 class groups} allow us to treat all the cases together.
\end{rmk}

\begin{proof}
Our aim is to find all the invariant ideals $I \subset \Oo_E$ such that
\begin{equation} \label{inclusion rachel}
F_{K3, I}(E) \subset K(E).
\end{equation}
By Proposition \ref{better understanding 2} we have a diagram of field extensions 
\begin{center}
\begin{tikzcd}[every arrow/.append style=dash]

& K_{I}(E) 
 \arrow{dd} & \\
 F_{K3,I}(E) 
  \arrow{dr} & \\
& K_{I}'(E)
 \arrow{d}\\
& E 

\end{tikzcd}

\end{center}
however, since $E$ is quadratic imaginary, we also have that $K_{I}'(E) =  F_{K3,I}(E) $. Introducing the Hilbert class field in the diagram above we obtain 
\begin{center}
\begin{tikzcd}[every arrow/.append style=dash]

& K_{I}(E) 
 \arrow{dd} & \\
K(E) 
  \arrow{ddr}  \arrow{ur} & \\
& K_{I}'(E)
 \arrow{d}\\
& E 

\end{tikzcd}

\end{center}
so that $$\Gal( K_{I}(E)  / K(E)) \cong \ker \pi ,$$
where $\pi$ is the canonical projection $$\pi \colon \Cl_{I}(E) \to \Cl(E),$$
and $$\Gal( K_{I}(E)  / K'_{I}(E)) \cong \Cl_{I}(E)^G.$$
Hence, the inclusion \eqref{inclusion rachel} becomes 
\begin{equation} \label{inclusion galois}
\ker \pi \subset \Cl_{I}(E)^G .
\end{equation}
Using the fundamental exact sequence 
\begin{equation}\label{fes}
1 \rightarrow \Oo^{I}_E \rightarrow \Oo^{\times}_{E} \rightarrow (\Oo_E / I)^{\times} \rightarrow \Cl_{I}(E) \rightarrow \Cl(E) \rightarrow 1,
\end{equation}

we see that $$ \ker \pi  \cong  \big( \mathcal{O}_{E}/ I \big)^{\times} / \mu(E).$$ 
It follows that \eqref{inclusion galois} holds true if and only if $G$ acts trivially on $\big( \mathcal{O}_{E}/ I \big)^{\times} / \mu(E).$
\end{proof}
\begin{rmks}
\leavevmode
\begin{itemize}
\item In particular, for every product of the form $\mathfrak{r} = \mathfrak{r}_1 \cdots \mathfrak{r}_k$ with $\mathfrak{r}_i$ distinct ramified primes, we see that $\Oo_E / \mathfrak{r}$ is a possible Brauer group for a principal K3 surface $X / K(E)$ with CM by $E$. 
\item It is a consequence of Theorem 2.6. and Theorem 3.1. of \cite{MR3483120} that if $\mu(E) = \{ \pm 1\}$ and $\ell \geq 3$ is a prime of $\Z$ that does not ramify in $E$, then $\Br(\overline{X})[\ell^\infty]^{G_K} = 0$. We note that under this assumption
$\Gal(E / \Q)$ does not act trivially on $(\Oo_E / \ell^n)^{\times} / \{ \pm 1 \}$ if $n > 0$ and $\ell >3$. If $\ell = 3$, then two things can happen (still assuming that it does not ramify): if $3$ splits in $E$, then $\Oo_E / 3$ is a possible Brauer group for a K3 surface $X / K(E)$ with CM by $E$ (this does not contradict Newton's result, but it is taking into account all the other K3 surfaces $X$ which are not the Kummer surface of a product of the same elliptic curve), whereas if $3$ is inert, we still have that $\Gal(E / \Q)$ does not act trivially on $(\Oo_E / \ell^n)^{\times} / \{ \pm 1 \}$ for every $n > 0$.
\end{itemize}
\end{rmks}

\bibliographystyle{plain}
\bibliography{bibliographyT}
\end{document}